\pdfoutput=1
\documentclass[11pt]{amsart}

\usepackage{fullpage, graphicx, enumitem, mathtools,stmaryrd}
\usepackage{url}
\usepackage{verbatim}
\usepackage{bbm}
\usepackage{amssymb,amsmath}
\usepackage{amsfonts,savesym,bm,amsthm}
\usepackage{xcolor}
\usepackage[mathscr]{eucal}
\usepackage[all]{xy}

 \usepackage{import}
 \usepackage{xifthen}
 \usepackage{pdfpages}
 \usepackage{transparent}

\definecolor{hot}{RGB}{65,105,225}
\usepackage[pagebackref=true,colorlinks=true, linkcolor=hot ,  citecolor=hot, urlcolor=hot]{hyperref}

\newtheorem{theorem}{Theorem}[section]
\newtheorem{lemma}[theorem]{Lemma}
\newtheorem{conjecture}[theorem]{Conjecture}
\newtheorem{theorem-definition}[theorem]{Theorem-Definition}
\newtheorem{corollary}[theorem]{Corollary}
\newtheorem{proposition}[theorem]{Proposition}

\newtheorem{definition-theorem}[theorem]{Definition-Theorem}

\newtheorem{theorem-defintion}[theorem]{Theorem-Definition}
\newtheorem{corollary-definition}[theorem]{Corollary-Definition}
\newtheorem{definition-proposition}[theorem]{Definition-Proposition}

\theoremstyle{definition}
\newtheorem{example}[theorem]{Example}
\newtheorem{definition}[theorem]{Definition}
\newtheorem{defn}[theorem]{Definition}

\newtheorem{remark}[theorem]{Remark}

\newtheorem{subs}[theorem]{ }

\numberwithin{equation}{section}

\def\bC{\mathbb{C}}
\def\be{\begin{equation}}
	\def\ee{\end{equation}}

\def\bN{\mathbb{N}}
\def\bZ{\mathbb{Z}}
\def\al{\alpha}
\def\bQ{\mathbb{Q}}

\def\xa{\xrightarrow}

\def\cL{\mathcal L}

\def\cA{\mathcal A}

\def\bA{\mathbb A}
\def\Zmot{Z^{\mathrm{mot}}}
\def\bL{\mathbb L}
\def\Var{\mathrm{Var}}
\def\al{\alpha}
\def\Ztd{Z^{\mathrm{td}}}
\def\Bir{\text{Bir}}
\def\sX{\mathscr X}
\def\sXtd{{\mathscr X}^{\mathrm{td}}}
\def\codim{\mathrm{codim}\;}
\def\sY{\mathscr Y}
\def\rat{\rho}
\def\Zbir{Z^{\mathrm{bir}}}
\def\cB{\mathscr{B}}
\def\TT{\llbracket T\rrbracket}
\def\td{\mathrm{td}}
\def\cR{\mathcal R}
\def\cD{\mathcal D}
\def\cS{\mathcal S}
\def\lct{\mathrm{lct}}
\def\bR{\mathbb{R}}
\def\dlct{{}^{(d)}\lct}
\def\VP{\mathrm{VP}}

\def\hkappa{\hat{\kappa}}
\def\hr{\hat{r}}
\def\tr{\mathrm{tr}}
\def\lcm{\mathrm{lcm}}

\title{A remarkable subset of poles of the motivic zeta function}

\author{Nero Budur}
\address{Department of Mathematics, KU Leuven, Celestijnenlaan 200B, 3001 Leuven, Belgium;  YMSC, Tsinghua University, 100084 Beijing, China;  BCAM, Mazarredo 14, 48009 Bilbao, Spain.}
\email{nero.budur@kuleuven.be}

\author{Eduardo de Lorenzo Poza}
\address{Department of Mathematics, KU Leuven, Celestijnenlaan 200B, 3001 Leuven, Belgium;  BCAM, Mazarredo 14, 48009 Bilbao, Spain.}
\email{eduardo.delorenzopoza@kuleuven.be}

\author{Quan Shi}
\address{Department of Mathematical Sciences, Tsinghua University, Beijing, 100084, P. R. China.}
\email{shiq24@mails.tsinghua.edu.cn / thusq20@gmail.com}

\author{Huaiqing Zuo}
\address{Department of Mathematical Sciences, Tsinghua University, Beijing, 100084, P. R. China.}
\email{hqzuo@mail.tsinghua.edu.cn}

\begin{document}

	\begin{abstract} For any polynomial $f$ with complex coefficients we find a remarkable subset of poles of the motivic zeta function. It is combinatorially determined by any log resolution and it admits an intrinsic interpretation in terms of contact loci of $f$. This uncovers a new, unexpected difficulty with proving the monodromy conjecture.
	\end{abstract}
	
	\maketitle
	
	\tableofcontents
	
	\section{Introduction}\label{introduction}
	
		Let $X$ be an $n$-dimensional smooth complex irreducible variety and $f:X\to\bA^1$ a regular function with non-empty zero locus. The monodromy conjecture of Igusa, Denef, and Loeser predicts that the poles of the motivic zeta function $\Zmot_f(T)$  give rise to local monodromy eigenvalues of $f$ and to roots of the $b$-function of $f$. It relates motivic aspects of $f$, which include some arithmetic aspects, with the complexity of singularities of $f$.
		Any log resolution of $f$ produces a list of candidate poles of $\Zmot_f(T)$. However, most candidate poles are not actual poles, and there is a lack of understanding of the end result after cancellations. 
		
			One wishes to single out  more tractable subproblems. One such attempt is to address residues in sufficiently generic situations. This leads to a conjecture of Denef, Jacobs, and Veys on motivic principal value integrals, see \cite{BSZ} as well as the newer Conjecture 1.7 there. Unfortunately, this conjecture seems equally difficult.

	In this article we find a remarkable subset of  poles of $\Zmot_f(T)$, determined combinatorially by any log resolution. In fact, we have many interpretations of these numbers. At the moment we cannot show that these combinatorial poles give rise to local monodromy eigenvalues or roots of the $b$-function. However, this subproblem of the monodromy conjecture seems more approachable. Moreover, this uncovers a new, unexpected  difficulty with proving the monodromy conjecture, since until now it was widely believed that the main difficulty is understanding what being a pole of $\Zmot_f(T)$ means. Here we have a set of guaranteed poles admitting alternative interpretations from many desired angles. The monodromy conjecture for these remarkable poles is  a problem fully within classical singularity theory, and even a problem within Floer theory.				
		
		 To state the main result, recall that by \cite[Theorem 2.2.1]{DL98} the motivic zeta function $Z_{f}^{\mathrm{mot}}(T)$ can be computed via a log resolution in the following way. Let $\mu : Y \to X$ be a log resolution of $f$ that is an isomorphism over $X\setminus f^{-1}(0)$. Write $\mu^*(\mathrm{div}(f)) = \sum_{i\in S} N_i E_i$ with $N_i\in \bZ_{>0}$ and $E_i$ mutually distinct prime divisors on $Y$, and let  $K_\mu = \sum_{i\in S} (\nu_i-1) E_i$ be the relative canonical divisor. By assumption,  $\cup_{i\in S}E_i$ has simple normal crossings. Then 
	\begin{equation}\label{Zmot}
		Z_{f}^{\mathrm{mot}}(T) = \sum_{I\subset S} [E_I^\circ] \cdot \prod_{i\in I} \frac{(\mathbb L-1)\mathbb L^{-\nu_i}T^{N_i}}{1-\mathbb L^{-\nu_i}T^{N_i}} \quad\quad\in K_0(\Var_\bC)[\bL^{-1}]\llbracket T \rrbracket,
	\end{equation}
	where $E_I^{\circ} = (\cap_{i\in I} E_i) \setminus (\cup_{j\in S\setminus I} E_j)$, $[V]$ is the class of a variety $V$ in the Grothendieck ring of varieties $K_0(\Var_\bC)$, and $\bL=[\bA^1]$. The expression is independent of the choice of log resolution for $f$.	
			Let $N_I := \mathrm{gcd}_{i\in I} N_i$ and $\al_I := \min_{i\in I}\{\nu_i/N_i\}$ for $\emptyset \neq I \subset S$ with $E_I \neq \emptyset$.
		
	\begin{defn}\label{defRmkble}  A number $\alpha \in \mathbb Q_{>0}$ is called {\it remarkable (with respect to $f$)}  if there exists $\emptyset \neq I \subset S$ with $E_I \neq \emptyset$ and $\al=\al_I$ such that: if $\al_J<\al_I$ for some  $\emptyset \neq J \subset S$ with  $E_J \neq \emptyset$, then  $N_{J}$ does not divide $N_{I}$.
		\end{defn}

	\begin{theorem}\label{ThmAA}
		Let $X$ be an $n$-dimensional smooth complex irreducible variety and $f:X\to\bA^1$ a regular function with non-empty zero locus. Then:
\begin{enumerate}
\item	 The set of remarkable numbers is independent of the choice of log resolution for $f$.
\item If $\al$ is a remarkable number, then $-\al$ is a pole of $\Zmot_f(T)$.
\end{enumerate}
\end{theorem}

		It is not difficult to see that the log canonical threshold, $\mathrm{lct}(f):=\min_{i\in S}\{\nu_i/N_i\}$, is always a remarkable number. So the remarkable numbers are generalizations of $\lct(f)$, defined purely in terms of the combinatorics and the vanishing orders associated with a log resolution of $f$. Other known invariants of $f$ usually depend on extra topological or geometric data. For example, the jumping numbers of $f$ are defined using the sequence of multiplier ideals $f$, and the spectral numbers are defined using the Hodge filtration on the cohomology of the Milnor fibers of $f$. The remarkable numbers are however not necessarily jumping numbers or spectral numbers, see Example \ref{exJN}. 
		One can have arbitrarily many remarkable numbers already in the case of plane curves, see Example \ref{exmany}. 	On the other hand, for log canonical hypersurfaces, non-degenerate hypersurfaces, and reduced hyperplane arrangements, $\lct(f)$ is the only remarkable number, see Section \ref{Examples}.

		The monodromy conjecture for the remarkable numbers would thus produce combinatorially some roots of the $b$-function $b_f(s)$ of $f$ and  monodromy eigenvalues of $f$: 
		
\begin{conjecture}\label{conjMCR} With notation as in Theorem \ref{ThmAA}, if $\al$ is a remarkable number of $f$, then $b_f(-\al)=0$ and $e^{-2\pi i\al}$ is a monodromy eigenvalue on the nearby cycles complex of $f$.
\end{conjecture}

An intrinsic interpretation of the remarkable numbers is as follows. Recall first that the intrinsic definition of the motivic zeta function is
\begin{equation}\label{Zmot0}
		Z_{f}^{\mathrm{mot}}(T) = \sum_{m \geq 0} [\sX_m(f)] \mathbb L^{-mn}  T^{m} 
	\end{equation}
where $\sX_m(f)$ is the $m$-contact locus of $f$, consisting of $m$-jets on $X$ with  order of contact with $f$ equal to $m$.
Motivated by the new concept of birational zeta functions from \cite{Birational zeta function}, we introduce the \textit{top-dimension zeta function} 
\be\label{edZtd} 
Z_{f}^{\mathrm{td}}(T) := \sum_{m\geq 0} \{\sX_m^{\mathrm{td}}(f)\}  \mathcal L^{-mn} T^m \quad\quad\in \bZ[\Bir_\bC][\cL^{-1}]\llbracket T\rrbracket.
\ee
Here $\bZ[\Bir_\bC]$ is the Burnside ring of \cite{KT}, whose underlying abelian group is the free abelian group on birational equivalence classes of  complex irreducible varieties, \{V\} denotes the sum of classes in $\bZ[\Bir_\bC]$ of the irreducible components of a variety $V$, $\cL=\{\bA^1\}$, and $\sXtd_m(f)\subset \sX_m(f)$ is the union of the top-dimensional components of $\sX_m(f)$.
The zeta function $\Ztd_f(T)$ admits  a rational function expression in terms of a log resolution, see Lemma \ref{realization_as_td_of_a_p_series}. Theorem \ref{ThmAA} is  a consequence of:

 \begin{theorem}\label{ThmA}
		Let $X$ be an $n$-dimensional smooth complex irreducible variety and $f:X\to\bA^1$ a regular function with non-empty zero locus. Let $\al\in\bQ$. Then:
		\begin{enumerate}
\item	 	 $-\alpha$ is a pole of $Z_{f}^{\mathrm{td}}(T)$ if and only if $\al$ is a remarkable number.
	\item If $-\alpha$ is a pole of  order $r>0$ of $Z_{f}^{\mathrm{td}}(T)$, then it is a pole of order $\geq r$ of $Z_f^{\mathrm{mot}}(T)$. 
\end{enumerate}		
	\end{theorem}

		We can also determine  combinatorially the orders of the poles of $Z^\td_f(T)$ but the expression is complicated even for the pole order of $-\lct(f)$, see Theorem \ref{thmPoleOrds}. The proofs are completely elementary. They rely on the signed graded ring structure on $\bZ[\Bir_\bC]$.
		
Note that there is no useful specialization from $K_0(\Var_\bC)$ to $\bZ[\Bir_\bC]$ since the latter is graded by dimensions. We also do not know how to lift $\Zmot_f(T)$ as a rational function over the localization of a dimension-graded version of $K_0(\Var_\bC)$. In fact this latter problem is essentially equivalent to the Embedded Nash Problem, which asks to characterize geometrically the irreducible components of the $m$-contact loci $\sX_m(f)$.

The top-dimension zeta function $Z^\td_f(T)$ was our initial focus and this is how we came across the remarkable numbers. It was a surprise for us that other numbers than the log canonical threshold can arise this way. Later we found  other interpretations of the remarkable numbers. The first one is that $Z^\td_f(T)$ and $\Zmot_f(T)$ admit a  common specialization which has the same poles and orders as $Z^\td_f(T)$, see (\ref{eqVP}). So  one can prove Theorem \ref{ThmAA} by skipping Theorem \ref{ThmA}. However, this would only simplify the notation and not the argument, which is already elementary. Since $Z^\td_f(T)$ captures something more about the contact loci than the specialization from (\ref{eqVP}), we opted to keep the current presentation.

	An easy reinterpretation of the remarkable numbers is as follows.
	
	\begin{defn}
	Let $d\in\bN_{>0}$. Fix a log resolution of $f$ as above. The {\it $d$-log canonical threshold} of $f$ is
	$$
	{}^{(d)}\lct(f):=\min\{\al_I\mid \emptyset\neq I\subset S\text{ with } E_I\neq\emptyset \text{ and }N_I\text{ divides }d\}.
	$$
	\end{defn}
	\noindent For $d$ big and divisible enough, this is the usual $\lct(f)$. If $f$ is not reduced, then $\dlct(f)$ might not exist.

\begin{lemma}\label{lemdlct} $\;$
\begin{enumerate}
\item Each $d$-log canonical threshold is independent of the choice of log resolution.
\item  The set of remarkable numbers of $f$ is the set $\{\dlct(f)\mid d\in\bN_{>0}\}$.
\end{enumerate}
\end{lemma}

It is easy to see that $\dlct(f)$ is periodic in $d$ with a period dividing  $N= \lcm(N_i\mid i\in S)$. 

Here is another intrinsic interpretation of remarkable numbers in terms of the contact loci. It  refines \cite[Corollary 0.2]{Mus02} for $\lct(f)$. Let  $C_m$ be the codimension of $\sX_m(f)$ inside the $m$-jet space of $X$, with the convention that the empty set has codimension $\infty$.

	\begin{theorem}\label{intrinsic_of_td_poles}
		With notation as above, let $N = \lcm(N_i\mid i\in S)$.  Let $1\leq d \le N$. Then:
        \begin{enumerate}
            \item $\dlct(f)$ exists if and only if $\sX_{lN+d}(f)\neq\emptyset$ for some $l\ge 0$.
            \item If $\dlct(f)$ exists, then the sequence $C_{lN+d}/(lN+d)$ decreases as $l$ increases and  $$\lim_{l\to\infty}C_{lN+d}/(lN+d)=\dlct(f).$$
        \end{enumerate}
	\end{theorem}




The remarkable numbers can also be defined locally at a point $x\in f^{-1}(0)$, and then a remarkable number of $f$ is a local remarkable number at some point $x$, see Remark \ref{locally_remarkbale}. 
Hence a better definition of the set of remarkable numbers of $f$ would be the union over $x\in f^{-1}(0)$ of the sets of local remarkable numbers.

If the hypersurface defined by  $f$ has  an isolated singularity at $x$, the local remarkable numbers admit an alternative characterization in terms of the Floer cohomology $HF^*(\phi^m,+)$ of the iterates of the monodromy $\phi$ on the Milnor fiber of $f$ at $x$. This refines \cite[Corollary 1.4]{McL19} for $\lct(f)$.

\begin{theorem}\label{intrinsic_of_td_poles_2}
	With notations as above, assume $f$ has an isolated singularity at $x\in X$. Then the set of  remarkable numbers of $f$ at $x$ is
	\begin{displaymath}
		\big\{  \lim_{l \to \infty}\big( \mathrm{min}\{v_{lN+d}+1,1\}  \big) \mid 1 \leq d\le N \big\}
	\end{displaymath}
	where $v_{m} = \inf\{-\frac{j}{2m} \mid HF^j(\phi^m,+) \neq 0\}$. In particular, the remarkable numbers are embedded contact invariants of the link of $f$ at $x$.
\end{theorem}
	
It is known that the invariant $v_m$ can be read from the first page of the spectral sequence constructed in \cite{McL19} that converges to 	$HF^*(\phi^m,+)$. By the proof of \cite[Theorem 1.1]{BP}, this spectral sequence is constant in $\mu$-constant families, hence we obtain:

\begin{theorem} If $f_t:(\bC^n,0)\to (\bC,0)$ is a $\mu$-constant $t$-continuous family of holomorphic function germs with isolated  singularities, then the set of remarkable numbers of $f_t$ is independent of $t$. 
\end{theorem}

This generalizes the constancy of $\lct(f_t)$ in $\mu$-constant families due to Varchenko \cite{Va}. An analog of Zarsiki's multiplicity conjecture for log canonical thresholds was posed in \cite{Bsurvey}, asking if $\lct(f)$ is an embedded topological invariant for germs of reduced hypersurfaces. We conjecture here further that the set of remarkable numbers of a reduced hypersurface germ $f$ is an embedded topological invariant.

\subs{\bf Plane curves.} We determine completely the remarkable numbers in the case of unibranch plane curve singularities:

\begin{theorem}\label{thmCurves} Let $f:(\bC^2,0)\to (\bC,0)$ define the germ of a non-smooth  unibranch plane curve singularity. Let $\mu$ be the minimal log resolution of $f$. Then the set of remarkable numbers of $f$ is the set of quotients $\al_i=\nu_i/N_i$ for $E_i$ meeting exactly three other $E_j$ with $j\in S\setminus\{i\}$.
\end{theorem}

The divisors $E_i$ meeting precisely three (or more, in the multibranch case) other $E_j$  are called the {\it rupture divisors} of the plane curve singularity. In the unibranch case, 
the number of rupture divisors is the genus of the singularity, equal to the number of Puiseux pairs, and it can be arbitrarily high. 

For a unibranch plane curve singularity, the negatives $-\al_i$ for rupture divisors $E_i$ are exactly all the poles different than $ -1$ of $Z^{mot}_f(s)$, they are roots of $b_f(s)$, and $e^{-2\pi i \al}$ are exactly all the monodromy eigenvalues  of $f$ different than $1$, see Loeser \cite{L}.

Recall that the Monodromy Conjecture is true for any plane curve singularity by \cite{L}. Hence Conjecture \ref{conjMCR} is true also for any plane curve singularity by Theorem \ref{ThmAA}. In fact, we know the poles of the motivic zeta function in this case, see \cite[Theorem 4.3]{Ve}: they are the negatives of $\al_i$ for rupture vertices and for strict transforms of irreducible components of the reduced zero locus of $f$. Hence the remarkable numbers are $\al_i$ for the same type of components $E_i$  by Theorem \ref{ThmAA}. On the other hand, even for reduced multibranch plane curve singularities, the analog of Theorem \ref{thmCurves} is not true. That is,  all remarkable numbers still come from rupture divisors, but  some rupture divisors might not contribute with remarkable numbers in the multibranch case, see Example \ref{exMultibr}.	

In Section \ref{Examples} we give  other examples besides plane curves.

\subs{\bf Real remarkable numbers.}\label{subsReal}  Let now $X_\bR$ be a smooth quasi-projective scheme of finite type over $\bR$, and $f_\bR:X\to \bA^1_\bR$ a morphism defined over $\bR$.  The $m$-contact loci of $f_\bR$ are defined over $\bR$, and \cite{DL98} define the motivic zeta function  $\Zmot_{f_\bR}(T)\in K_0(\Var_\bR)[\bL^{-1}]\TT$, where $K_0(\Var_\bR)$ denotes the Grothendieck ring of reduced separated quasi-projective schemes of finite type over $\bR$, and $\bL=[\bA^1_\bR]$. The formula (\ref{Zmot}) is also proved in \cite{DL98}  for $\Zmot_{f_\bR}(T)$ in terms of a log resolution $\mu_\bR:Y_\bR\to X_\bR$ of $f_\bR$. Here $E_I$ with $I\subset S$ are defined over $\bR$.

We assume  now that $X(\bR)$ is non-empty and consider it with the induced structure of a real algebraic manifold, see \cite{Man}. We denote by $f:X(\bR)\to\bR$ the  real algebraic function induced by $f_\bR$, and we assume it is non-constant. A real  log resolution of $f$ is induced from $\mu_\bR$ by taking the real loci, see \cite{S-real}. Let $S'=\{i\in S\mid E_i(\bR)\neq\emptyset\}$.  Then the divisor defined by $f$ on the real  log resolution is $\sum_{i\in S'}N_i E_i(\bR)$. 
		
	\begin{defn}\label{defRmkbleReal}  A number $\alpha \in \mathbb Q_{>0}$ is called {\it real remarkable} with respect to the real algebraic function $f$,  if there exists $\emptyset \neq I \subset S'$ with $E_I(\bR) \neq \emptyset$ and $\al=\al_I$ such that: if $\al_J<\al_I$ for some  $\emptyset \neq J \subset S'$ with  $E_J(\bR) \neq \emptyset$, then  $N_{J}$ does not divide $N_{I}$.
		\end{defn}

The real log canonical threshold $\mathrm{rlct}(f):=\min_{i\in S'}\{\nu_i/N_i\}$ is a real remarkable number of $f$. We have the real analog of Theorem \ref{ThmAA}:

	\begin{theorem}\label{ThmAAReal}
		Let $f$ be non-constant real algebraic function on a real algebraic manifold. Then:
\begin{enumerate}
\item	 The real remarkable numbers of $f$ are independent of the choice of log resolution.
\item If $\al$ is a real remarkable number, then $-\al$ is a pole of $\Zmot_{f_\bR}(T)$.
\end{enumerate}
\end{theorem}

\subs{\bf Outline.}
In Section \ref{secPre} we recall some preliminary material, and give the definition of  poles of  birational and motivic zeta functions. In Section \ref{secTD} we introduce the top-degree subseries for a formal power series over a signed graded ring. The algorithm from this section is applied in Section \ref{sectdzeta} to prove the main theorems from this introduction. The remaining results from the introduction are also proved in Section \ref{sectdzeta}, except the plane curve case which is deferred to Section \ref{secCurves}. In Section \ref{Examples} give  other examples.
	
	\subs{\bf Acknowledgement.} We thank Yifan Chen for Example \ref{exNondeg} and D. Bath, G. Blanco, and M. Gonz\'alez Villa for discussions. N. Budur was supported by the KU Leuven Grant Methusalem METH/21/03  and  FWO Grant G0B3123N. E. de Lorezo Poza was supported by FWO Grant 1187425N. Q. Shi, and H. Zuo were supported by BJNSF Grant 1252009. H. Zuo was supported by NSFC Grant 12271280.

		\section{Preliminary material}\label{secPre}

	\subs{\bf Notation.} A {\it ring} in this paper is always assumed to be associative, commutative, with a unity. A {\it variety} is always reduced, but possibly not irreducible and not equidimensional.
A {\it log resolution} $\mu:Y\to X$ of a pair $(X,D)$, where $D$ is an effective divisor on a complex irreducible smooth variety $X$, is always assumed to be an isomorphism above $X\setminus D$, beside the rest of the usual assumptions.

	\subs{\bf Contact loci.}\label{subs_contact_loci} Let $X$ be a complex irreducible variety. For  $m\in\bN$ we denote by $\mathcal L_{m}(X)$ the {\it $m$-jet space} of $X$, that is, the set of morphisms $\gamma : \mathrm{Spec}\, \mathbb C[t]/(t^{m+1}) \to X$ over $\bC$. It is naturally endowed with the structure of $\bC$-scheme of finite type, but we only consider it with its induced reduced structure.  
The {\it arc space} $\mathcal L_{\infty}(X)$ is the set of arcs in $X$, that is, morphisms $\gamma : \mathrm{Spec}\, \mathbb C\llbracket t\rrbracket \to X$, endowed with the natural structure of reduced $\bC$-scheme. This is the same as the inverse limit $\bC$-scheme $\cL_{\infty}(X)=\varprojlim \cL_m(X)$ obtained from the natural truncation morphisms $\phi_{m,m'}:\cL_m(X)\to\cL_{m'}(X)$ for $m\ge m'$. There are natural truncations morphisms $\phi_{\infty, m}:\cL_{\infty}(X)\to\cL_m(X)$.
A morphisms of varieties $\mu:Y\to X$ induces morphisms $\mu_l:\cL_l(Y)\to \cL_l(X)$ compatible with the truncations maps. For $\gamma\in\cL_l(X)$ with $l\in\bN\cup\{\infty\}$, denote by $\gamma(0):=\phi_{l,0}(\gamma)\in X$ the   center of $\gamma$.

Let now $X$ be a smooth complex irreducible variety  and $f:X\to \bA^1$ a regular function such that the scheme-theoretic zero locus $D:=f^{-1}(0)$ is non-empty. Since $X$ is smooth, the truncation morphisms $\phi_{m,m'}$ are Zariski locally-trivial $\bA^{n(m-m')}$-fibrations.
For $m \in \mathbb \bZ_{>0}$ and $l \in \mathbb N_{\geq m} \cup \{\infty\}$, we define the {\it $m$-contact locus} of $f$ in $\cL_l(X)$:
	\begin{displaymath}
		\sX_m^l(f) := \{\gamma \in \mathcal L_l(X) \mid \mathrm{ord}_\gamma f = m\} = \phi_{l,m-1}^{-1}(\mathcal L_{m-1}(D))\setminus \phi_{l,m}^{-1}(\mathcal L_{m}(D)),
	\end{displaymath}
 where $\mathrm{ord}_{\gamma} f$ is the order in $t$ of $f$ via $\gamma$, and set $\sX_0^l=\phi_{l,0}^{-1}(X\setminus D)$.
 If $l=m$ we set $\sX_m(f):=\sX_m^l(f)$. Then $\sX^l_m(f)=\phi_{l,m}^{-1}(\sX_m)$.

	Let $\mu : Y \to X$ be a log resolution of $f$. We keep the notation as in the introduction before (\ref{Zmot}). Let $m\in\bZ_{>0}$.
Assume that the log resolution $\mu$ is  {\it $m$-separating}, that is, $N_i+N_j > m$ for all $E_i\cap E_j \neq \emptyset$. Let $S_m = \{i\in S \mid N_i\text{ divides }m\}$. For $i\in S_m$ and  $l\in\bZ_{\ge m}\cup\{\infty\}$, let 
$$
\sY_{m,i}^l:=\{\tilde\gamma\in \sX_{m/N_i}^l(f\circ \mu)\mid \tilde\gamma(0)\in E_i^\circ\}\subset\cL_l(Y),
$$
$$\sX_{m,i}^l:=\mu_l(\sY_{m,i}^l)\subset\sX_m^l(f).$$
For the following, see \cite[Theorem 3.4]{Cohomology_of_Contact_Loci} and \cite[Proposition 3.2]{BBPP24}.

	\begin{theorem}\label{structure_thm_of_contact_loci}
Let  $X$ be a smooth complex irreducible variety  and $f:X\to \bA^1$ a regular function with non-empty zero locus. Let $\mu : Y \to X$ be a log resolution of $f$. Fix an $m \in\bZ_{>0}$ and suppose $\mu$ is $m$-separating. Let $l\in\bZ_{\gg m}\cup\{\infty\}$. Then:

\begin{enumerate}
\item  $\phi_{l',l}(\sX^{l'}_{m,i})=\sX_{m,i}^l$ and  $\sX^{l'}_{m,i} = \phi_{l',l}^{-1}\sX_{m,i}^l$ for all $l'\in \bZ_{\ge l}\cup\{\infty\}$ and $i\in S_m$.
\item $\sX^l_{m,i}$ is smooth, non-empty, irreducible, and locally closed in $\cL_l(X)$ for $i\in S_m$.
\item There is a partition $\sX_m^l(f) = \bigsqcup_{i\in S_m} \sX_{m,i}^l$.
\item The restricted morphism $\mu_l : \sY_{m/N_i}^l \to \sX_{m,i}^l$ is a Zariski locally-trivial $\bA^{m(\nu_i-1)/N_i}$-fibration.
\item The morphism $\sY_{m/N_i}^l \to E_i^\circ$, $\tilde\gamma \mapsto \tilde\gamma(0)$, is a Zariski locally-trivial $\bA^{nl-m/N_i} \times \bC^*$-fibration.
	\end{enumerate}
	\end{theorem}

	\subs{\bf Grothendieck ring.}\label{Grothendieck_ring}
	The Grothendieck group of complex varieties $K_0(\Var_\bC)$ is defined to be the quotient of the free abelian group generated by isomorphism classes of complex varieties by the relations
		$[X]=[U]+[X\setminus U]$ for $U$ an open subset of a complex variety $X$. It is a ring by defining
		 $[X] \cdot [Y] := [X\times Y]$ for two complex varieties. 
		 Recall that $\mathbb L := [\bA^1]$.
		 
The {\it virtual Poincar\'e specialization} is the ring homomorphism $\mathrm{VP} : K_0(\Var_\bC) \to \mathbb Z[w]$ defined by  $$X\to\sum_{m,i}  (-1)^{i+m} \dim \mathrm{Gr}_m^W H_c^i(X,\mathbb C) \cdot w^{m}$$ 
	for varieties $X$,
	where $W$ is the weight filtration. For a variety $X$, $\mathrm{VP}([X])$ is a polynomial of degree $2\dim X$ in $w$ and the coefficient of $w^{2\dim X}$ equals the number of top-dimensional irreducible components of $X$. 	In particular, $\mathrm{VP}(\mathbb L) = w^2$. 
	We will also denote by $\mathrm{VP}:K_0(\Var_\bC)[\bL^{-1}]\to \mathbb Z[w,w^{-1}]$ the ring homomorphism induced by inverting $\bL$ and $w$, respectively.

\subs{\bf Burnside ring.}	\label{subsBurn}
		Let $\mathrm{Bir}_{\mathbb C}^d$ be the set of birational equivalence classes of $d$-dimensional  complex irreducible varieties. Denote by $\{X\}$ the birational equivalence class of an irreducible variety $X$.
	Let $\bZ[\Bir^d_\bC]$ be the free abelian group generated by $\Bir^d_\bC$.	
		The \textit{Burnside ring of $\mathbb C$} is defined to be the additive group $\bZ[\Bir_\bC]:= \bigoplus_{d\geq 0} \bZ[\Bir^d_\bC]$ further endowed with a multiplication defined as $\{X_1\} \cdot \{X_2\} := \{X_1 \times X_2\}$, cf. \cite{KT}. This is a graded ring. If $X=\cup_{i=1}^rX_i$ is the irreducible decomposition of a variety $X$, set $\{X\} := \sum_{i=1}^r \{X_i\}$ in $\bZ[\Bir_\bC]$. 
		Let $\mathcal L := \{\bA^1\}$.

The {\it rational specialization} is the ring homomorphism $\rat  : \bZ[\Bir_\bC] \to \mathbb Z[\mathcal L]$ defined by sending the class of an irreducible variety $\{X\}\in\Bir^d_\bC$  to $\mathcal L^{d}$. Let
$$
\cB_\cL:=\bZ[\Bir_\bC][\cL^{-1}]. 
$$
We will also denote by $\rat:\cB_\cL \to \mathbb Z[\mathcal L,\mathcal{L}^{-1}]$ the ring homomorphism induced by inverting $\cL$. 

There is also the {\it point specialization} $\pi :\cB_\cL\to \mathbb Z$ mapping the birational class of any irreducible variety to $1$. Note that $\pi$ factors through $\rat$.

	\subs{\bf Birational zeta functions.} Let  $X$ be a smooth complex irreducible variety  and $f:X\to \bA^1$ a regular function with non-empty zero locus. Let $\mu:Y\to X$ be a log resolution of $f$. We keep the notation as in the introduction before (\ref{Zmot}). 
	The \textit{birational zeta function of $(f,\mu)$} is defined as
	\begin{equation}\label{eqZfmubir}
		Z_{f,\mu}^{\mathrm{bir}}(T) := \sum_{I\subset S} \{E_I\} \cdot \prod_{i\in I} \frac{\mathcal L \cdot \mathcal L^{-\nu_i} T^{N_i}}{1-\mathcal L^{-\nu_i}T^{N_i}}\quad\quad\in\cB_\cL\llbracket T\rrbracket.
	\end{equation}
	Our definition here is different by a constant term from that in \cite{Birational zeta function}. In  \cite{Birational zeta function} it was assumed that $f^{-1}(0)$ is reduced,  $\mu$ was allowed more generally  to be a dlt resolution, and it was shown that  $\Zbir_{f,\mu}(T)=\Zbir_{f,\mu'}(T)$ if $\mu$ and $\mu$ and $\mu'$ are two crepant-birationally equivalent dlt resolutions. The notation  $\Zbir_f(T)$ is reserved in \cite{Birational zeta function} for the birational zeta function of a dlt modification of a reduced $f$, in which case it admitted an intrinsic interpretation in terms of the contact loci $\sX_m(f)$.

\subs{\bf Poles.}\label{subPoles}  The notion of a pole of a rational function as in (\ref{eqZfmubir}) has to be defined with care since $\cB_\cL$ is not a domain. 
We follow \cite[\S 4]{RV} to formalize these definitions, but our final definition will be slightly different.

\begin{remark}\label{rmkRvpoles} The notion of a pole of the rational specialization of a function as in (\ref{eqZfmubir}) is more or less intuitive since $\bZ[\cL^\pm]\simeq \bZ[x^\pm]$ is a domain. However, even in this case there is a choice to make. For example, consider $Z(T)=(1-x^{-1}T)(1-x^{-2}T^2)(1-x^{-3}T^3)^{-1}$. One could declare that $Z(T)$ has no
pole since $1-x^{-1}T$ appears with total order 0; this is what the definition of a pole from \cite[\S 4]{RV} is based on. One could also declare that $Z(T)$ has a pole at $T=x$ of order 1, since 1 is the maximum order as pole of $\xi x$, after the base change from $\bZ$ to $\bC$, where $\xi$ is a root of unity; this definition will be our choice. This latter choice is motivated by Lemma \ref{rationality_non-positive}, which does not hold for the former definition. 
\end{remark}

Let $\cD$ be the multiplicatively closed subset  of $\cB_\cL[T]$ generated by $\{1-\cL^{-a}T^b\mid a, b\in\bN_{>0}\}$. By abuse of notation we shall also denote the corresponding subset of $\bZ[\cL^\pm][T]$ by $\cD$. We will define poles of elements of $\cB_\cL[T][\cD^{-1}]$ and of $\bZ[\cL^\pm][T][\cD^{-1}]$ by specializing the elements further.

Set 
$$\cB[\cL^\bQ]:=\cB_\cL[\cL^{1/N}\mid N\in\bN_{>0}],\quad \cA=\cB[\cL^\bQ][(1-\cL^q)^{-1},b^{-1}\mid q\in\bQ\setminus\{0\},b\in\bZ\setminus\{0\}],$$
$$\bZ[\cL^\bQ]:=\bZ[\cL^\pm][\cL^{1/N}\mid N\in\bN_{>0}],\quad A=\bZ[\cL^\bQ][(1-\cL^q)^{-1},b^{-1}\mid q\in\bQ\setminus\{0\},b\in\bZ\setminus\{0\}].$$
Let $\cD_\bQ$ be the multiplicatively closed subset  of $\cB[\cL^\bQ][T]$ generated by $\{1-\cL^{-q}T^b\mid q\in\bQ, b\in\bN_{>0}\}$. By abuse of notation we shall also denote by $\cD_\bQ$ 
 the corresponding subsets of  $\cA[T]$,  $\bZ[\cL^\bQ][T]$, and $A[T]$. 
Set 
$$\cR:= \cB[\cL^\bQ][T][\cD_\bQ^{-1}],\quad \cR':=\cA[T][\cD_\bQ^{-1}],$$
$$R:= \bZ[\cL^\bQ][T][\cD_\bQ^{-1}],\quad R':=A[T][\cD_\bQ^{-1}].$$
Consider the commutative diagram of natural ring homomorphisms
\be\label{eqComDia}
\xymatrix{
\cB_\cL\TT \ar[d]^\rho & \cB_\cL[T][\cD^{-1}] \ar[r] \ar[d] \ar@{_{(}->}[l]  & \cR \ar[r] \ar[d]  & \cR' \ar[d]\\
\bZ[\cL^\pm]\TT& \bZ[\cL^\pm][T][\cD^{-1}] \ar@{^{(}->}[r] \ar@{_{(}->}[l] & R \ar@{^{(}->}[r] & R'
}
\ee
where the  vertical maps are induced by the rational specialization $\rho$, and the bottom horizontal maps are  injective maps of integral domains. 

\begin{defn} Let $q\in \bQ$.

(1) If $0\neq H(T)\in A[T]$, respectively $\in \cA[T]$, denote by $n(H(T),q)$ the unique natural number $n$ such that $(T-\cL^{-q})^n\mid H(T)$ but $(T-\cL^{-q})^{n+1}\nmid H(T)$ in $A[T]$, respectively in $\cA[T]$. This is well-defined \cite[4.6]{RV}.

(2) If $0\neq Z(T)\in R'$ (respectively $\in \cR'$), write $Z(T)=F(T)/G(T)$ with $0\neq F(T)\in A(T)$ (respectively  $\in\cA[T]$) and $G(T)\in \cD_\bQ\subset A(T)$ (respectively  $\subset\cA[T]$). We call $q$ (or $\cL^{-q}$) {\it  an RV-pole} of $Z(T)$ if $n(G(T),q)-n(F(T),q)>0$. If $q$ is an RV-pole, we call $n(G(T),q)-n(F(T),q)$ {\it the order of $q$ as an  RV-pole.} (The minus sign convention here for the pole $q$ and $\cL^{-q}$ is for traditional reasons, since in motivic integration one usually evaluates $T=\bL^{-s}$.)
\end{defn}

It is easy to see that this definition is independent of the choice of representation $Z(T)=F(T)/G(T)$ if $Z(T)\in R'$ since in this case we are working with rational functions over a domain. By \cite[4.12]{RV}, the proof of which goes word by word here too, the same conclusion holds for $Z(T)\in \cR'$.

\begin{defn}\label{defRVpoles}
If $0\neq Z(T)\in \cB_\cL[T][\cD^{-1}]$, respectively $\in \bZ[\cL^\pm][T][\cD^{-1}]$, we say $q\in \bQ$ is {\it an RV-pole of $Z(T)$ of order $n$} if it is so
for  the image of $Z(T)$ in $\cR'$, respectively in $R'$.
\end{defn}

Note that the RV-pole orders for $Z(T)\in \cB_\cL[T][\cD^{-1}]$ under the specialization $\rho$ might decrease. That is, an RV-pole of order $n$ of $\rho(Z(T))$ is an RV-pole of $Z(T)$ of order at least $n$.

\begin{remark}
The above definition mirrors the definition of a pole and its order for the motivic zeta function $\Zmot_f(T)$ from \cite[\S 4]{RV}. 
\end{remark}

Our working definition of poles will be slightly altered, cf. Remark \ref{rmkRvpoles}. Redefine from now:	

$$\cA:=\bC\otimes \cB[\cL^\bQ] [(1-\xi\cL^q)^{-1}\mid q\in\bQ\setminus\{0\}, \xi\text{ is a root of unity}],$$
$$A:=\bC\otimes \bZ[\cL^\bQ] [(1-\xi\cL^q)^{-1}\mid q\in\bQ\setminus\{0\}, \xi\text{ is a root of unity}],$$
and redefine $\cR'$ and $R'$ accordingly as above. Then the following terms are also well-defined, by adapting directly the proofs in \cite{RV}.

\begin{defn}\label{defPole} Let $q\in \bQ$ and let $\xi\in\bC$ be a root of unity.

(1) If $0\neq H(T)\in A[T]$, respectively $\in \cA[T]$, denote by $n(H(T),q,\xi)$ the unique natural number $n$ such that $(T-\xi\cL^{-q})^n\mid H(T)$ but $(T-\xi\cL^{-q})^{n+1}\nmid H(T)$ in $A[T]$, respectively in $\cA[T]$. 

(2) If $0\neq Z(T)\in R'$ (respectively $\in \cR'$), write $Z(T)=F(T)/G(T)$ with $0\neq F(T)\in A(T)$ (respectively  $\in\cA[T]$) and $G(T)\in \cD_\bQ\subset A(T)$ (respectively  $\subset\cA[T]$). We call $q$ (or $\cL^{-q}$) {\it  a pole} of $Z(T)$ if $M_q:=\max_{\xi}(n(G(T),q,\xi)-n(F(T),q,\xi))>0$, where the maximum is over the roots of unity $\xi$. If $q$ is a pole, we call $M_q$ {\it the order of $q$ as a pole.} 

(3) If $0\neq Z(T)\in \cB_\cL[T][\cD^{-1}]$, respectively $\in \bZ[\cL^\pm][T][\cD^{-1}]$, we say $q\in \bQ$ is {\it a pole of $Z(T)$ of order $n$} if it is so
for  the image of $Z(T)$ in $\cR'$, respectively in $R'$.
\end{defn}	
	
 We  give next an equivalent formulation of the definition of poles and pole orders. Note that any element $Z(T)$ of $\cB_\cL[T][\cD^{-1}]$ (resp. $\bZ[\cL^\pm][T][\cD^{-1}]$, $\cR'$, $R'$) can be written as
$$
Z(T) = {Q(T)}/{\prod_{i=1}^r (1-\cL^{-c_i}T^N)^{e_i}}
$$
for some $r, N, c_i, e_i \in \mathbb Z_{>0}$, $c_1 < \dots < c_r$, and $Q(T)$ in $\cB_\cL[T]$ (resp. $\bZ[\cL^\pm][T]$, $\cA[T]$, $A[T]$). We will call this a {\it standard expression} for $Z(T)$.

\begin{lemma}\label{estimation of poles}
	Let $Z(T) = {Q(T)}/{\prod_{i=1}^r (1-\cL^{-c_i}T^N)^{e_i}}$ in $\cR'$ (resp. $R'$) be a standard expression.
	
(1)	Then the poles of $Z(T)$ are contained in $\{-c_i/N\}_{i=1,\dots,r}$, with the order of $-c_i/N$ as a pole $\leq e_i$.

(2) If $Q(T)$ is not divisible by $1 - \mathcal{L}^{-c_i}T^N$ in $\cA[T]$ (resp. $A[T]$) then $-c_i/N$ is a pole of order exactly $e_i$.

(3) If $Q(T)$ is not divisible by $1 - \mathcal{L}^{-c_i}T^N$ in $\cA[T]$ (resp. $A[T]$) for all $i$, then for any other standard expression $Z(T) = Q'(T)/\prod_{j=1}^{r'} (1- \cL^{-c'_j}T^{N'})^{e'_j}$  and for any $i=1,\dots, r$, there exists $j$ such that $c_i/N = c'_j/N'$ and $e_i \leq e'_j$.
\end{lemma}
\begin{proof} (1)
	 Follows directly from the definition. (2) Let $\xi_j = e^{2\pi ij/N}$ with $j=0,...,N-1$. Suppose the order of $-c_i/N$ as a pole is less than $e_i$, then $Q(T)$ is divisible by $T-\xi_j \mathcal L^{c_i/N}$ for all $j$ in $\cA[T]$ (resp. $A[T]$). Since $T-\xi_j \mathcal L^{c_i/N}$ are pairwise coprime  in $\cA[T]$ (resp. $A[T]$) for all $j$, it follows that $Q$ is divisible by $\prod_{j=0}^{N-1} (T-\xi_j \mathcal L^{c_i/N}) = T^N - \mathcal L^{c_i}$ in $\cA[T]$ (resp. $A[T]$), a contradiction. 
	 
	 (3) Let $\tilde N$ be a common multiple of $N$ and $N'$. Then $Z(T) = \tilde Q(T) / \prod_{i=1}^r(1- \cL ^{-\tilde Nc_i/N}T^{\tilde N})^{e_i}$, where $\tilde Q = Q\cdot \prod_{i=1}^r(\sum_{j=0}^{\tilde N/N-1}  \cL^{-c_i j}T^{Nj})^{e_i}$. Then $\tilde Q$ is not divisible by each $1-\cL^{-\tilde Nc_i/N}T^{\tilde N}$ since $1- \cL^{-\tilde Nc_i/N}T^{\tilde N}$ for all $i$ are pairwise coprime in $\cA[T]$ (resp. $A[T]$). So we may assume $N = N'$. Applying coprimality again, one deduces that for each $i=1,...,r$, there exists $j\in \{1,\dots,r'\}$ such that $c_i = c'_{j}$ and $e_i \leq e'_{j}$.
	 \end{proof}

\begin{remark} (1) Lemma \ref{estimation of poles} is not necessarily true over $\cB_L[T][\cD^{-1}]$ since the coprimality claim in the proof is not necessarily true over $\cB_L[T]$.

(2) 
Lemma \ref{estimation of poles} is not true for RV-poles, see Definition \ref{defRVpoles}. Take $Z(T)=(1-x^{-1}T)(1-x^{-3}T^2)(1-x^{-3}T^3)^{-1}$ as in Remark \ref{rmkRvpoles}.  Then $-1$ is the a pole of order $1$ of $Z(T)$, but $-1$ is not an RV-pole of $Z(T)$.
\end{remark}

\begin{corollary}
Let $Z(T)$ and $q$ be as in Definition \ref{defPole}. Then $q$ is a pole of order $n$ of $Z(T)$ as in Definition \ref{defPole} if and only is it a pole of order $n$ as in the following definition.
\end{corollary}

\begin{definition} Let  $q\in\bQ$. Let $Z(T)\in \cR'$ (resp. $R'$). We say {\it $q$ is a pole of $Z(T)$ of order at most $n$} if there exists a standard expression  $Z(T) = {Q(T)}/{\prod_{i=1}^r (1-\cL^{-c_i}T^N)^{e_i}}$ in $\cR'$ (resp. $R'$) such that
 $(c_i,e_i) = (Nq,n)$ for some $i$. We say {\it $q$ is a pole of order $n$ of $Z(T)$} if and only if it has order at most $n$, but not at most $n-1$.  If $ Z(T)\in \cB_\cL[T][\cD^{-1}]$ (resp. $\in \bZ[\cL^\pm][T][\cD^{-1}]$), we say $q\in \bQ$ is {\it a pole of $Z(T)$ of order $n$} if it is so
for  the image of $Z(T)$ in $\cR'$, respectively in $R'$.
\end{definition}

We apply similar modifications to define  poles and their orders for $\Zmot_f(T)$, and   for any other specializations from $\cB_\cL\TT$ and $K_0(\Var_\bC)[\bL^{-1}]\TT$. We do not repeat these here. We note that pole orders can only decrease under specializations.

	\section{Top-degree specializations}\label{secTD}

In this section we develop the top-degree specialization operator for formal power series over signed graded rings. This is motivated by the fact that the localized Burnside ring $\cB_\cL$ admits a signed graded ring structure and we will apply the results from here in Section \ref{sectdzeta} to the top-dimension zeta function $Z^\td_f(T)$.

	\begin{definition}\label{deftdop}
		Let $R = \bigoplus_{d \in \mathbb Z} R_d$ be a graded ring, with $R_d$ the degree $d$ component. Let  $a \in R$. We denote by $a^{\mathrm{td}}$ the top-degree part of $a\neq 0$. We set $0^{\mathrm{td}} = 0$. Define the degree $\deg(a)$ to be the degree of $a^{\td}$. 
		If $A = \sum_{i\geq 0} a_i T^i \in R\llbracket T\rrbracket$, we define the \textit{top-degree subseries of $A$} to be $A^{\mathrm{td}} := \sum_{i\geq 0} a_i^{\mathrm{td}}\, T^i$.

		The graded ring $R$ is called \textit{signable} if there exists a sequence of additively closed subsets $P_d \subset R_d$, $d \in \mathbb Z$, such that: $1\in P_0$, $0\notin P_d$, {$P_d\neq\emptyset$ if $R_d\neq\{0\}$,}
		and $P_{d}\cdot P_e \subset P_{d+e}$ for all $d,e \in \mathbb Z$. Then $(R,\{P_d\}_{d\in\bZ})$ is called a \textit{signed graded ring}. In this case, the element $0\neq a\in R$ is called \textit{positive} if $a$ can be expressed as $\sum_{i=1}^r a_i$, where each $a_i \in P_{d_i}$ for some $d_i \in \mathbb Z$, and some $r\in\bN_{>0}$. The element $a$ is called \textit{subpositive} if $a^{\mathrm{td}}$ is positive. 
	Denote by  $R_+$ and $R_{\#}$  the subsemirings of $R$ consisting of $0$ and the positive and, respectively, subpositive elements. We call $R_+$ and $R_{\#}$  the \textit{positive} and, respectively, {\it subpositive semiring}.
	\end{definition}
	
	\begin{example}\label{canonincal_sign_of_B} Denote for simplicity
		 $$\cB_d := \bZ[\Bir^d_\bC], \quad \cB:=\bZ[\Bir_\bC]= \oplus_{d\in \mathbb Z} \cB_d, \quad \cB_\cL:=\bZ[\Bir_\bC][\cL^{-1}].$$  
		 Let $\bN[\Bir_\bC^d]\subset \cB_d$ be the free abelian monoid generated by birational equivalence classes of complex irreducible varieties. Define
		  $\cB_d^+ := \bN[\Bir_\bC^d] \setminus \{0\}$. Then $(\cB,\{\cB_d^+\}_d)$ is a signed graded ring with $\cB_+=\oplus_d\bN[\Bir_\bC^d]$.

		  The localization $\cB_\cL$ has an induced graded ring structure.
		  Define  $P_d := \cup_{l\in \mathbb Z} B_{l+d}^+ \cdot \mathcal L^{-l}$. 
		Using the point specialization $\pi:B_\cL \to \mathbb Z$ induced by sending the birational equivalence class $\{X\}$ of a complex irreducible variety to  $1$, one finds that $0 \notin P_d$ for all $\in \mathbb Z$. So $P_d$ is also equal to $(\cup_{l\in\bZ}\bN[\Bir_\bC^{d+l}]\cdot\cL^{-l})\setminus\{0\}$.
		 Then $(B_{\mathcal L}, \{P_d\}_d)$ is a signed graded ring.
	\end{example}

	We now fix a signed graded ring $(R,\{P_d\})$, with positive and semipositive semirings $R_+$ and $R_\#$.
	
	\begin{definition}
		 A formal power series of the form $$\frac{B}{ \prod_{i=1}^r (1-a_iT^{d_i})} \in R_\#\llbracket T\rrbracket$$ 
		where $B\in R_\#[T]$, $a_1,\dots,a_r\in R_+$ are homogeneous and positive, and $d_1,\dots,d_r > 0$, is called a \textit{basic rational subpositive series}. A \textit{rational subpositive series} is finite sum of basic rational subpositive series. A \textit{rational quasi-subpositive series} is the sum of a polynomial in $R[T]$ and a rational subpositive series. 	
		\end{definition}

The following are easy observations:	

	\begin{lemma}\label{sum_of_td_part}\label{summation_of_rational quasi-subpositive series}$\;$
	\begin{enumerate}
\item	If $A = B\cdot \prod_{i=1}^r (1-a_iT^{d_i})^{-1}\in R_\#\TT$ is a basic rational subpositive series, then $A^{\mathrm{td}} = (B^{\mathrm{td}}\cdot \prod_{i=1}^r (1-a_iT^{d_i})^{-1})^{\mathrm{td}}\in R_+\TT$.
	
\item	If $A_1, A_2 \in R_\#\llbracket T\rrbracket$, then $(A_1+A_2)^{\mathrm{td}} = (A_1^{\mathrm{td}}+A_2^{\mathrm{td}})^{\mathrm{td}} \in R_+\llbracket T\rrbracket$.

\item		If $A_1, A_2\in R\TT$ are two rational quasi-subpositive series, then $(A_1+A_2)^{\mathrm{td}} = (A_1^{\mathrm{td}}+A_2^{\mathrm{td}})^{\mathrm{td}} + W$ for some $W \in R[T]$.
\end{enumerate}
	\end{lemma}

	We will be interested in computing the poles of the top-degree subseries of a rational quasi-subpositive series. We start with basic rational subpositive series, but first we need a technical definition.

	\begin{definition}
		Let $Q = \sum_{i\geq 0} q_iT^i \in R[T]$ be a polynomial and fix $0\neq c\in \bZ$, $d\in \mathbb Z_{>0}$. We define the polynomial $Q_{c,d}^1 = \sum_i \bar q_i T^i \in R[T]$ by setting 
		
		$$\bar q_i := \begin{cases}
			0 & \text{ if } q_i = 0;\\
			0 & \text{ if } q_i\neq 0\text{ and }\exists\, j\neq i\text{ such that }q_j\neq 0,\, d\text{ divides }i-j,\\
			 & \text{ and }\deg(q_i)-\deg(q_j)<\left(\lfloor \frac{i}{d}\rfloor-\lfloor \frac{j}{d}\rfloor\right)c;\\
			q_i^\td & \text{ otherwise}.
		\end{cases}$$

We define inductively the polynomial $Q_{c,d}^r := (Q-\sum_{j=1}^{r-1} Q_{c,d}^j)^1_{c,d}$ for $r>1$.	
	\end{definition}

%
%
%
%
%
	
\begin{example}
If $Q=\sum_{i=0}^{10}q_iT^i\in R[T]$, all $q_i\neq 0$, and all $\deg(q_i)$ are equal, then $Q^1_{2,3}=q_0^\td+q_1^\td T+q_2^\td T^2$. If in addition $q_0$, $q_1$, $q_2$ are homogeneous, then $Q^2_{2,3}=q_3^\td T^3 + q_4^\td T^4+q_5^\td T^5$. If in addition $q_3$, $q_4$, $q_5$ are homogeneous, then $Q^3_{2,3}=q_6^\td T^6+q_7^\td T^7+q_8^\td T^8$. If in addition $q_6$, $q_7$, $q_8$ are homogeneous, then $Q^4_{2,3}=q_9^\td T^9+q_{10}^\td T^{10}$.
\end{example}	
	
	\begin{remark}
		Suppose there exists an element $0\neq a \in R$  homogeneous of degree $c$. Fix a residue $j\in \{0,\ldots, d-1\}$ modulo $d$. Consider the subseries $Q_j$ and $(Q_{c,d}^1)_j$ of $Q$ and $Q_{c,d}^1$, respectively, given by the  powers of $T$ with residue class $j$ modulo $d$. Then $(Q_{c,d}^1)_j$ is formed from $Q_j$ as follows: replace $T^d$ with $aT^d$; 
		keep as the only non-zero coefficients of the powers of $T$, the coefficients of the highest (as the power of $T$ varies) degree ; finally, revert $aT^d$ back to $T^d$.
				\end{remark}

	\begin{lemma}\label{td_of_basic_rational subpositive series}
		Let $$A = \frac{B}{ \prod_{i=1}^r (1-a_iT^{d_i})}\in R_\#\TT$$ be a basic rational subpositive series. Order the indices such that
 $$\alpha := \frac{\deg (a_1)}{d_1} = \dots = \frac{\deg (a_m)}{d_m} > \frac{\deg (a_{m+1})}{d_{m+1}}\geq \dots \geq \frac{\deg (a_r)}{d_r}.$$
 Take $d\in\bN$ such that	 $d_1d_2\dots d_r$ divides $d$. Let  $Q = B \cdot \prod_{i=1}^r (1-a_i^{d/d_i}T^d)(1-a_iT^{d_i})^{-1}$.
 Then $$A^{{td}} = \frac{Q_{d\alpha,d}^1}{\prod_{i=1}^m(1-a_i^{d/d_i}T^{d})} + W$$ for some $W \in R[T]$. Moreover, we have $Q_{d\alpha,d}^1 \in R_+[T]$.
	\end{lemma}
	\begin{proof}
		Note that $Q\in R_\#[T]$ and  $A={Q}\cdot{\prod_{i=1}^r (1-a_i^{d/d_i} T^d)^{-1}}$.  Thus every basic rational subpositive series can be written in such  a way that the powers $T^{d_i}$ in the denominator are replaced by one single power $T^d$ at the expense of changing the coefficients $a_i$. Note however that  the set of ratios $\deg (a_i)/d_i$ does not change.	
		
		Expand $A$ as a formal power series and consider the contribution of each monomial term to $A^{\mathrm{td}}$. The first observation is that $\prod_{i>m}(1-a^{d/d_i}T^d)^{-1}$ contributes with just multiplication by $1$ to $A^{\mathrm{td}}$. 
That is,  $$A^{\mathrm{td}} = ({Q}\cdot{\prod_{i=1}^m (1-a_i^{d/d_i} T^d)^{-1}} )^{\mathrm{td}}=(Q\cdot\prod_{i=1}^m\sum_{k\ge 0}(a_i^{d/d_i} T^d)^k)^\td.$$ Indeed, for $i>m$, by substituting $a_i^{d/d_i}T^d$ with $a_1^{d/d_i}T^d$
in any monomial $(a_1^{d/d_1} T^d)^{k_1}\ldots (a_m^{d/d_m} T^d)^{k_m}$, the degree of the coefficient strictly increases.

By Lemma \ref{sum_of_td_part}, to compute $A^\td$ we can and do replace $Q$ with $Q^\td$. Note that the claim does not change, since $Q^1_{c,d}=(Q^\td)^1_{c,d}\in R_+[T]$ for any $0\neq c\in\bZ$.

Next note that $A^\td=\sum_{w=0}^{d-1}A_w^\td$, where the sum is over residue classes $w$ modulo $d$ and $A_w$ is the subseries of $A$ given by the powers of $T$ with residue class $w$ modulo $d$. Hence it suffices to show the claim for each $A_w$, namely, 
 it suffices to consider the case $Q = T^w \cdot Q_0(T^d)=\sum_{l\in\bN}q_{w+ld} \cdot T^{w+ld}\in R_+[T]$ for some $w\in\{0,\dots,d-1\}$, $Q_0 \in R_+[T]$, and $q_{w+ld}$ is homogeneous and positive if non-zero. In this case, for $k$ strictly bigger than the degree in $T$ of $Q(T)$, the term $q_{w+ld} \cdot T^{w+ld}$ of $Q$, where $q_{w+ld}\neq 0$, has non-zero contribution to the power $T^{w+kd}$ in $A^{\mathrm{td}}$  if and only if 
 \be\label{eqConQ}\deg (q_{w+ld}) + (k-l)d\alpha = \max\{\deg (q_{w+l'd}) + (k-l')d\alpha \mid l'\in\bN,   q_{w+l'd}\neq 0 \}.\ee
 Note that here we are using the signed graded ring structure on $R$ to guarantee that the contributions to the coefficient of $T^{w+kd}$ in $A^\td$ do not cancel each other.  
Note also that the condition (\ref{eqConQ}) is independent of $k$. The claim now follows by the definition of $Q_{d\al,d}^1$, and the polynomial $W$ in the claim appears since we only looked at $k>\deg_TQ(T)$. 
	\end{proof}

	Next step computes the top-degree subseries for sums of outputs from Lemma \ref{td_of_basic_rational subpositive series}.

	\begin{lemma}\label{td_of_summation_of_basic_rational subpositive series}
		For $i\in\{1,\ldots, p\}$ let $$A_i = \frac{b_i T^{w_i + l_i d}}{\prod_{j=1}^{m_i}(1-a_{ij}T^{d})}\in R_+\TT$$ where
	\begin{itemize}	
\item		$ l_i\in \mathbb N$, $d, m_i\in\bN_{>0}$, $w_i \in \{0,\dots,d-1\}$,
\item		 $b_i, a_{ij}\in R_+$ are non-zero homogeneous and positive, 
\item		 $\deg (a_{i1}) = \ldots =\deg (a_{im_i})$ for each $i$.
\end{itemize} 
Define $\al_i:=\deg (a_{i1})/d$ and   $\lambda_i := \deg (b_i) - l_i\alpha_id$.
 For $w\in\{0,\ldots, d-1\}$ let
\begin{itemize}
\item $\cS_w := \{i \mid w_i = w\}$, $\cS_w' := \{i\in \cS_w \mid \alpha_i = \max_{k\in \cS_w}\alpha_{k}\}$,  $\cS_w'' := \{i\in \cS_w' \mid \lambda_i = \max_{k\in \cS_w'} \lambda_k\}$. 
\end{itemize} 
Then $(\sum_{i=1}^{p} A_i) ^{\mathrm{td}} = W + \sum_{w = 0}^{d-1} \sum_{i \in \cS_w''} A_i$ for some $W\in R[T]$.
	\end{lemma}
	\begin{proof}
		Consider the contributions of each series $A_i$ to $(\sum_{i=1}^p A_i)^{\mathrm{td}}$. It suffices to look at the case $w_1 = w_2 = \dots = w_n = w$. Let $\alpha = \max_{i\in \cS_w}\alpha_i$ and $\lambda = \max_{i\in \cS_w'} \lambda_i$. Note that for $N\gg 0$,  the top-degree component of the coefficient of $T^{w+Nd}$ in $\sum_{i=1}^p A_i$  receives a contribution only from those $A_i$ with $i\in \cS_w'$. Hence $(\sum_{i=1}^p A_i)^{\mathrm{td}}=(\sum_{i\in \cS_w'} A_i)^{\mathrm{td}}$ up to  addition of a polynomial in $R[T]$. Further analyzing the contribution of each $b_i T^{w_i + k_i d}$, we have $(\sum_{i\in \cS_w'} A_i)^{\mathrm{td}} = (\sum_{i\in \cS_w''} A_i)^{\mathrm{td}}$ up to  addition of a polynomial in $R[T]$. We complete the proof by noting that for all large $N \in \mathbb N$, coefficients of $T^{w+Nd}$ in all $A_i$ with $i\in S_w''$ are homogeneous and have the same degree $\lambda+N\alpha d$.
	\end{proof}


	\begin{proposition}\label{rational quasi-subpositive series-td}\label{td_remark_2}
		Let $(R,\{P_d\})$ be a signed graded ring. Then the top-degree subseries of any rational quasi-subpositive series is again a rational quasi-subpositive series.
	\end{proposition}
	\begin{proof}
		Lemmas \ref{summation_of_rational quasi-subpositive series},  \ref{td_of_basic_rational subpositive series}, and  \ref{td_of_summation_of_basic_rational subpositive series} together provide the algorithm to compute the top-degree subseries of a rational quasi-subpositive series, up to a polynomial, in terms of a finite sum of basic rational subpositives series.	
		\end{proof}
	
	When the base ring $R=\mathbb Z[x,x^{-1}]$, the rationality result as in Proposition \ref{rational quasi-subpositive series-td} can be extended to more general rational functions, but we lose some positivity in the conclusion.
	
	\begin{lemma}\label{rationality_non-positive}
		Let $R = \mathbb Z[x^\pm]$ endowed with the standard grading by the degree in $x$. 
	In $R\TT$ take $A = {Q}\cdot{\prod_{i=1}^r (1-x^{-c_i}T^N)^{-e_i}}$, where $r, N, c_i, e_i \in \mathbb Z_{>0}$, $c_1 < \dots < c_r$, and $Q \in R[T]$. Then  $A^{\mathrm{td}}$ is also an element of $R[T][{\prod_{i=1}^r (1-x^{-c_i}T^N)^{-e_i}}]$. In particular, the set of poles (see Definition \ref{defPole} with $\bZ[x^\pm]\simeq\bZ[\cL^\pm]$) as a rational function in $T$ of $A^{\mathrm{td}}$, counting multiplicities, is contained in the set of poles of $A$.
	\end{lemma}
	\begin{proof} As before, to compute $A^\td$ we can split the task according to residue classes $w\in \{0,1,\dots,N-1\}$ modulo $N$ of powers of $T$, and we can assume that $Q = T^wQ_0(T^N)$. To this end we can further assume that $w = 0$ and that $Q$ is not a multiple of  $1-x^{-c_i}T^N$ for any $i$. By definition, we have $Q = \sum_{j \geq 1} Q_{-c_1,N}^j$. By our assumption, there exists a smallest integer $j_0 > 0$ such that $Q_{-c_1,N}^{j_0}$ is not a multiple of $1-x^{-c_1}T^N$. Then by analyzing the degree of the coefficients of high powers of $T^N$ in the expansion of $A$ as a power series, we conclude that that $A^{\mathrm{td}} = {Q_{-c_1,N}^{j_0}}\cdot{(1-x^{-c_1}T^{N})^{-e_1}}$ up to addition of a polynomial in $R[T]$. This proves the first claim. When varying the residue class $w$ modulo $N$ and when canceling out terms in $Q$ with terms in the denominator of $A$, a different $c_i$ than $c_1$ might survive as the minimum. The last claim follows from Lemma \ref{estimation of poles}.
\end{proof}
	
\begin{remark}
Lemma \ref{rationality_non-positive} is not true for RV-poles, see Definition \ref{defRVpoles}. Take $Z(T)=(1-x^{-1}T)(1-x^{-3}T^2)(1-x^{-3}T^3)^{-1}$ as in Remark \ref{rmkRvpoles}. Then $Z(T)^\td=(1-x^{-1}T-x^{-2}T^2)(1-x^{-3}T^3)^{-1}$. Then $-1$ is the only pole of both $Z(T)$ and $Z(T)^\td$ and has order one in both cases. It also has order $1$ as an RV-pole of $Z(T)^\td$, but $-1$ is not an  RV-pole of $Z(T)$.
\end{remark}

		\section{Top-dimension zeta functions}\label{sectdzeta}

	In this section we prove all the results from the introduction.	
	
	Recall that we defined in (\ref{edZtd}) the top-dimension zeta function $Z^\td_f(T)\in\cB_\cL\TT$ from the top-dimensional components of the $m$-contact loci $\sX_m(f)$.
		 We will apply the results from the previous section to determine the poles of $Z^\td_f(T)$. 
		  First we  need to show that $Z^\td_f(T)$ is a rational function, an element of $\cB_\cL[\cD^{-1}]$, in the notation of Definition \ref{defPole}, so that the notion of pole is well-defined.

Fix a log resolution $\mu : Y \to X$ of $f$. We use  the notation premerging the definition (\ref{edZtd}). By \cite[Lemma 2.9]{Cohomology_of_Contact_Loci}, we can produce a series of log resolutions $\{\mu_m:Y_m\to X\}_{m\geq 0}$ of $f$ with $\mu_0 = \mu$ and $\mu_m$ factoring through $\mu_{m-1}$,  such that $\mu_m$ is $m$-separating for all $m\geq 0$. This is achieved by inductively blowing up intersections of divisors in the reduced pullback of $D=f^{-1}(0)$. 
	
\begin{lemma}\label{lemmun} For all $m\ge 0$, $\Zbir_{f,\mu}(T)=\Zbir_{f,\mu_m}(T)$. In particular,  $\Zbir_{f,\mu}(T)^{\mathrm{td}}=\Zbir_{f,\mu_m}(T)^{\mathrm{td}}$ for the top-degree operator from Definition \ref{deftdop} on $\cB_\cL\TT$. Moreover, $\Zbir_{f,\mu}(T)^{\mathrm{td}}$ is rational quasi-subpositive series in $\cB_\cL\TT$ lying in the subring $\cB_\cL[T][\cD^{-1}]$.
\end{lemma}	
\begin{proof}
Since blowing up a stratum of a simple normal crossings divisor produces a crepant-birationally equivalent pair by considering the reduced pullback divisors, we have that 	$(Y_m,\mu_m^*(D)_{red})$  is crepant-birationally equivalent to $(Y, D_{red})$. Then \cite[Proposition 2.6]{Birational zeta function}  gives the first claim. By Example \ref{canonincal_sign_of_B} the localized Burnside ring $\cB_\cL$ admits a natural signed graded ring structure with positive and subpositive semirings $\cB_{\cL,+}$ and $\cB_{\cL,\#}$. So we can apply the top-dimension subseries operator. The last claim follows by Proposition \ref{rational quasi-subpositive series-td}:
\end{proof}

Although $td$ in $Z^\td_f(T)$ stands top-dimension, $Z^\td_f(T)$ is the top-degree subseries of something:
	
	\begin{lemma}\label{realization_as_td_of_a_p_series}
		We have $Z_{f}^{\mathrm{td}}(T) = \Zbir_{f,\mu}(T)^{\mathrm{td}}$.
			\end{lemma}
	\begin{proof}
		The proof is similar to the proof of \cite[Theorem 1.4]{Birational zeta function}. 
		We use the notation from Theorem \ref{structure_thm_of_contact_loci}. 
		Fix $m \geq 0$. For the first claim it suffices to show that the coefficient of $T^m$ in $(Z_{D,\mu}^{\mathrm{bir}}(T))^{\mathrm{td}}$ equals $\{\sX_m^{\mathrm{td}}\}\cdot \mathcal L^{-mn}$.  If $m = 0$ then both coefficients are equal to $\{X\}$. Let now $m > 0$. By Lemma \ref{lemmun} we may assume $\mu = \mu_m$ is $m$-separating. By  Theorem \ref{structure_thm_of_contact_loci}  the coefficient of $T^m$ in $Z_{f,\mu}^{\mathrm{bir}}(T)$ as a power series is
		\begin{displaymath}
			\sum_{i\in S_m} \{E_i\} \cdot \mathcal L^{1-\nu_im/N_i} = \sum_{i\in S_m} \{\sX_{m,i}^l\} \cdot \mathcal L^{-nl},
		\end{displaymath}
		where $l \gg 0$ and $S_m = \{i\in S \mid N_i \text{ divides } m\}$. Since $\sX_{m,i}^l$ are locally closed irreducible subvarieties of $\sX_m^l$ and the irreducible components of $\sX^l_m$ and those of $\sX_m$ determine each other, Theorem \ref{structure_thm_of_contact_loci} (1) yields that $\big(\sum_{i\in S_m} \{\sX_{m,i}^l\} \cdot \mathcal L^{-nl}\big)^{\mathrm{td}} = \{\sX_m^{\mathrm{td}}\}\cdot \mathcal L^{-mn}$.
	\end{proof}

Hence $Z^\td_f(T)$ is a rational series in $\cB_\cL[T][\cD^{-1}]$. Thus it has well-defined poles and pole orders, and these are intrinsic invariants of $f$, independent of the choice of log resolution $\mu$, since $Z^\td_f(T)$ is intrinsic. We will determine these poles.

	\begin{lemma}\label{residue_class_of_exponent} 
We use  the notation premerging the definition (\ref{edZtd}). Let 
		  $\emptyset \neq I \subset S$ with $E_I\neq\emptyset$, and let $Z_I = \{E_I\}\mathcal L^{\vert I\vert} \cdot \prod_{i\in I} \frac{\mathcal L^{-\nu_i} T^{N_i}}{1-\mathcal L^{-\nu_i}T^{N_i}}\in\cB_{\cL,+}\TT$. Then 
		$$Z_I^{\mathrm{td}} = W + \frac{P}{(1-\mathcal L^{-N\al_I}T^{N})^{m_I}}$$
		for some $W \in B_{\mathcal L}[T]$ and some $P\in B_{\mathcal L,+}[T]$ such that the set of residues modulo $N$ of the powers of $T$ appearing in $P$ is exactly the set $\{0,N_{I},2N_{I},\dots,N-N_{I}\}$, where $N\in\bN_{>0}$ is fixed and divisible enough, for example   $N=\lcm(N_i\mid i\in S)$, and $m_I=\#\{i\in I\mid \al_I=\nu_i/N_i\}$. 
	\end{lemma}
	\begin{proof}
		We apply Lemma \ref{td_of_basic_rational subpositive series} to $A=Z_I$ which gives us  $P = Q_{-N\al_I,N}^1$ and $W$, where $$Q:= \{E_I^\circ\}\mathcal L^{\vert I\vert} \cdot \prod_{i\in I}\sum_{j=1}^{N/N_i} (\cL^{-\nu_i}T^{N_i })^j.$$  Looking at the proof of Lemma \ref{td_of_basic_rational subpositive series}, for each $w \in \{0,\ldots,N-1\}$, if $T^{w+kN}$ appears in $Q$ for some $k\geq 0$, then $T^{w+lN}$ appears in $P$ for some $l\geq 0$. Note that modulo $N$ we have 
$			\{\sum_{i\in I} j_i N_i \in \mathbb Z/N\mathbb Z \mid j_i \in \{1,\ldots,N/N_i\}\} = \{0,N_I, 2N_I, \ldots,N-N_I\} \subset \mathbb Z/N\mathbb Z.
$	
\end{proof}
	
Recall that we defined in \ref{Grothendieck_ring} the rational specialization $\rat:\cB_\cL\to \bZ[\cL^{\pm}]$ and the point specialization $\pi :  \cB_{\mathcal L} \to \mathbb Z$ sending the class of an irreducible variety of dimension $d$ to $\cL^d$ and to 1, respectively.
	 We extend these linearly to ring homomorphisms $\rat$ and $\pi$ on $\cB_\cL\TT$. Recall  Definition \ref{defPole} of poles and their orders.

\begin{lemma}\label{existing_of_poles} 
Let $Z = {Q}\cdot{(1-\mathcal L^{-b}T^N)^{-m}}\in \cB_\cL\TT$ with $0\neq Q \in \cB_{\mathcal L,+}[T]$, $m\in\bN$, and $b, N\in\bN_{>0}$.  Then $-b/N$ is a pole of order $m$ of $Z$ and of $\rho(Z)$.
\end{lemma}
	\begin{proof} 
	In general, the pole orders under the specialization $\rho$ might decrease. So it is enough to prove the claim for $\rho(Z)$. Hence we can replace $\cB_\cL$ with $\bZ[\cL^\pm]$ in the assumptions. Since the bottom part of the diagram (\ref{eqComDia}) consists of injective maps of integral domains, it is enough to show that the residue $Q(\cL^{b/N})$ is non-zero in the extension $\bZ[\cL^\bQ]$ of $\bZ[\cL^\pm]$.
		The point specialization $\pi$ extends to $\bZ[\cL^\bQ]$ by sending $\mathcal L^{1/N}$ to $1$ as well. 
	Denote the extended homomorphism again by $\pi$. Then $\pi(Q(\mathcal L^{b/N})) > 0$ since $0\neq Q \in \bZ[\cL^\pm]_+[T]$, and this is a contradiction. 	
\end{proof}
	
	The same argument can be applied to sums of terms as in the above lemma to conclude:
	
\begin{lemma}\label{lemDetPoles}
Let 
\be\label{eqZforPoles} 
Z = W+\sum_{i=1}^p\sum_{j=1}^{m_i} \frac{Q_{i,j}}{(1-\mathcal L^{-b_i}T^{N})^j}\in \cB_\cL\TT
\ee 
where  $p, N, m_i, b_i\in\bN_{>0}$; $b_i\neq b_{i'}$ if $i\neq i'$;  $W, Q_{i,j} \in \cB_{\mathcal L}[T]$, $0\neq Q_{i,m_i} \in\cB_{\cL,+}$. 
Then $Z$ and $\rho(Z)$ have the same set of poles, $-b_i/N$ with $i=1,\ldots, p$, and the same pole order $m_i$ for each $i$.
\end{lemma}

	\begin{subs}{\bf Proof of Theorem \ref{ThmA} (1).} \label{prfThmA1}
	First, we show that $Z^\td_f(T)$ can be written as in (\ref{eqZforPoles}).	
	By Lemma \ref{realization_as_td_of_a_p_series}, $Z_f^\td(T)=(\sum_{I\subset S}Z_I)^\td$, where 
	 $Z_I = \{E_I\}\mathcal L^{\vert I\vert} \cdot \prod_{i\in I} \frac{\mathcal L^{-\nu_i} T^{N_i}}{1-\mathcal L^{-\nu_i}T^{N_i}}\subset\cB_{\cL,+}\TT$.  By Lemma \ref{sum_of_td_part} (2), $(\sum_{I\subset S}Z_I)^\td=(\sum_{I\subset S}Z_I^\td)^\td$. By Lemma \ref{residue_class_of_exponent}, this is further equal to
	\be\label{eqBig}
Z^\td_f(T)=	\bigg( \cL^n + \sum_{\substack{\emptyset\neq I\subset S\\ E_I\ne \emptyset}} \left (W_I+\frac{P_I}{(1-\cL^{-N\al_I}T^N)^{m_I}}\right)  \bigg)^\td
	\ee 
for some $W_I\in\cB_\cL[T]$, $P_I\in \cB_{\cL,+}[T]$, where $m_I=\#\{i\in I\mid \al_I=\nu_i/N_i\}$ and $N\in\bN_{>0}$ is fixed and divisible enough, for example $N=\lcm(N_i\mid i\in S)$.
 Moreover, the set of residues modulo $N$ of the powers of $T$ appearing in $P_I$ is exactly the set $\{0,N_{I},2N_{I},\dots,N-N_{I}\}$.
		Applying Lemma \ref{sum_of_td_part} again to (\ref{eqBig}), we have
\be\label{eqBig2}
Z^\td_f(T)=	W_0+ \bigg(  \sum_{\substack{\emptyset\neq I\subset S\\ E_I\ne \emptyset}} \frac{P_I}{(1-\cL^{-N\al_I}T^N)^{m_I}}  \bigg)^\td
	\ee 		
for some $W_0\in\cB_\cL[T]$. Next, we apply Lemma  \ref{td_of_summation_of_basic_rational subpositive series} to take the top-degree subseries of the big sum. A first consequence is that 
\be\label{eqBig3}
Z^\td_f(T)=	W+	\sum_I \frac{Q_I}{(1-\cL^{-N\al_I}T^N)^{m_I}}
\ee
for some polynomials $W\in\cB_\cL[T]$ and $Q_I\in\cB_{L,+}[T]$, where the sum runs over a possibly smaller subset of $I$'s than in (\ref{eqBig2}).	Hence $Z^\td_f(T)$ can be written as in (\ref{eqZforPoles}). In particular the poles and their orders are the same for $Z^\td_f(T)$ and its rational specialization $\rho(Z^\td_f(T))$ by Lemma \ref{lemDetPoles}.

		Let
			$S_\alpha$  be the set of subsets $\emptyset \neq I \subset S$ such that $E_I\neq \emptyset$, $\al_{I} = \alpha$, and for all  $\emptyset \neq J \subset S$ with $E_J \neq \emptyset$ and $\al_{J} < \al_{I}$, we have $N_{J} \nmid N_{I}$. We need to show that $-\alpha$ is a pole of $Z_{f}^{\mathrm{td}}(T)$ if and only if $S_{\alpha}\neq \emptyset$.
		
		Suppose $S_{\alpha} = \emptyset$. If there is no $\emptyset \neq I\subset S$ such that $\al_I = \alpha$, then $-\alpha$ is  not even a candidate pole of $Z^\td_f(T)$ by (\ref{eqBig2}).
		Otherwise, for all $\emptyset \neq I \subset S$ with $\al_I = \alpha$ and $E_I\neq \emptyset$, there exists $\emptyset \neq J \subset S$ with $E_J \neq \emptyset$, $\al_J < \al_I$, and $N_J \mid N_I$. Then 
 $$\bigg(  \frac{P_I}{(1-\cL^{-N\al_I}T^N)^{m_I}} + \frac{P_J}{(1-\cL^{-N\al_J}T^N)^{m_J}} \bigg)^\td =  W_{IJ} + \frac{P_J}{(1-\cL^{-N\al_J}T^N)^{m_J}}$$
for some $W_{IJ}\in\cB_\cL[T]$ by  Lemma \ref{td_of_summation_of_basic_rational subpositive series}. Hence, all such $Z_I$ do not contribute  to $Z_{f}^{\mathrm{td}}(T)$ in (\ref{eqBig3}) by applying repeatedly Lemma \ref{sum_of_td_part} (3). Thus  $-\alpha$ that is not a pole.

		On the other hand, suppose $S_{\alpha} \neq \emptyset$. Choose $I \in S_\alpha$ such that for all $J\in S_\alpha$, $N_J\nmid N_I$ if $N_J\neq N_I$. Let $\mathcal S_I = \{J\in S_{\alpha} \mid N_I = N_J\}$. We now use the information we have on the residue classes modulo $N$ of powers of $T$ in $P_J$, and in particular, we use the residue class $N_J$ modulo $N$. Then we see that there exists some  term of the form $$\frac{q\cdot T^{N_J+kN}}{(1-\mathcal L^{-N\al_J}T^N)^{m_J}}$$ with $0\neq q\in B_{\mathcal L,+}$ and $k\geq 0$,  produced by some $Z_J$ with $J\in \mathcal S_I$ via Lemma \ref{td_of_summation_of_basic_rational subpositive series}, which has non-zero contribution to $Z_f^{\mathrm{td}}(T)$ when passing from (\ref{eqBig2}) to (\ref{eqBig3}). By Lemma \ref{lemDetPoles}, $-\alpha$ is a pole. \hfill $\Box$
	\end{subs}

		In the course of the previous proof we have also obtained:

	\begin{corollary}\label{poles_of_td_zeta} The zeta function
	$Z_{f}^\td(T)$ and its rational specialization of $\rho(Z_{f}^\td(T))$ have the same poles and pole orders.
	\end{corollary}

\subs{\bf Pole orders.} We can  determine combinatorially  the orders of the poles of $Z^\td_f(T)$. In the proof of Theorem \ref{ThmA} (1) we used only limited information from Lemma \ref{td_of_summation_of_basic_rational subpositive series}. Using all the information from that lemma, a directly application thereof gives the following generalization of Theorem \ref{ThmA} (1), which however will not be used in the paper.

\begin{defn} Take $N$ big and divisible enough, for example $N=\lcm(N_i\mid i\in S)$.
For $I \subset S$ with $E_I \neq \emptyset$ and $a = (a_i)_{i\in I} \in \prod_{i \in I} \{1, \ldots, N/N_i\}$ set:
\[
N_a := \sum_{i \in I} a_i N_i, \quad
\nu_a := \sum_{i \in I} a_i \nu_i, \]
\[
w_a := N_a - \lfloor N_a/N \rfloor N, \quad
r_a := \nu_a - \alpha_I   \lfloor N_a/N \rfloor N, \quad m_I := \#\{ i \in I \mid \nu_i/N_i = \alpha_I \}.
\]
\end{defn}

\begin{theorem}\label{thmPoleOrds} With the above notation, $-\alpha$ is a pole of $Z_f^{\mathrm{td}}(T)$ of order $\geq m$ if and only if:

\begin{itemize}
    \item there exist $I\subset S$ with $E_I \neq \emptyset$ and $a \in \prod_{i \in I} \{1, \ldots, N/N_i\}$ such that $\alpha_I = \alpha$, $m_I \geq m$, and
    \item there do not exist $J\subset S$ with $E_J \neq \emptyset$ and $b \in \prod_{j \in J} \{1, \ldots, N/N_j\}$ with $w_b = w_a$, and 
    \[
    \left( \alpha_J < \alpha_I \text{ or } (\alpha_J = \alpha_I, m>1, \text{ and } r_a < r_b) \right).
    \]
\end{itemize}
\end{theorem}

	\begin{subs}{\bf Proof of Theorem \ref{ThmAA} (1).} The independence of the set of remarkable numbers of $f$ of log resolution $\mu$ follows from Theorem \ref{ThmA} (1), since the poles of $Z^\td_f(T)$ are independent of $\mu$. $\hfill\Box$ 
	\end{subs}
	
	\begin{remark} ($j$-remarkable numbers)
	In the last part of the above proof we considered only the residue class $w=N_I$ modulo $N$ when we applied Lemma \ref{td_of_summation_of_basic_rational subpositive series}. Considering the other residue classes $w=jN_I$ modulo $N$ for $j\in \bN_{>0}$ and non-zero contributions $qT^{w+kN}/(1-\cL^{-N\al_J}T^N)^{m_J}$ to $Z^\td_f(T)$, we see that the supporting poles are the following numbers: a number $\al_I$ for $\emptyset\neq I\subset S$ with $E_I\neq\emptyset$ is called {\it $j$-remarkable} if 
	$$\al_I=\min\{\al_{J}\mid \emptyset\neq J\subset S, N_J\text{ divides }jN_I\}.$$
	Then the remarkable numbers are the 1-remarkable numbers. The $j$-remarkable numbers form a filtration:
	$$\{\text{remarkable numbers}\}\supset \{\text{2-remarkable numbers}\} \supset \{\text{3-remarkable numbers}\}\supset\ldots\supset \{\lct(f)\},
	$$
where $\lct(f)$ is the only $N$-remarkable number. We remark without going into the proof that the $j$-remarkable numbers are independent of $\mu$ as well.
	\end{remark}

\begin{remark}\label{locally_remarkbale} (Local remarkable numbers)

(1)	 A  number $\alpha \in \mathbb Q$ is called \textit{remarkable at $\Sigma$}, for a closed subset $\Sigma\subset X$, if there exists $\emptyset \neq I \subset S$ with $E_I\cap \mu^{-1}(\Sigma) \neq \emptyset$ and $\al_I=\al$ such that: if $\al_J<\al_I$ for some  $\emptyset \neq J \subset S$ with  $E_J \cap \mu^{-1}(\Sigma) \neq \emptyset$, then  $N_{J}$ does not divide $N_{I}$. The set of  remarkable numbers at $\Sigma$ coincides with the negative of the poles of $Z_{f|_U}^{{td}}(T)$ for a sufficiently small Zariski open neighborhood $U$ of $\Sigma$, and hence is also independent of the choice of log resolution $\mu$. One can similarly defined further refinements, the $j$-remarkable numbers at $\Sigma$. For $j$ big and divisible enough, the only $j$-remarkable number at $\Sigma$ is the local log canonical threshold $\lct_\Sigma(f)$.
	
(2)	 If $X = \cup_{i} U_i$ is an open covering and $f_i=f_{|U_i}$, it is not difficult to see that the remarkable numbers of $f$ are contained in the union of the remarkable numbers of the $f_i$. Note also that the poles of $Z^\td_f(T)$ lie in the union of the set of  poles of $Z^\td_{f_i}(T)$ by Theorem \ref{ThmA} (1). 				
	\end{remark}

\subs{\bf Proof of Theorem \ref{ThmA} (2).} \label{subPfthmA}
 We now prove that the if $-\al$ is a pole of $Z^\td_f(T)$ of order $r>0$ then it is a pole of $\Zmot_f(T)$ of order $\ge r$. 
Extend linearly the virtual Poincar\'e specialization from \ref{Grothendieck_ring} to a ring map $\mathrm{VP}:K_0(\Var_\bC)[\bL^{-1}]\TT\to \bZ[w^\pm]\TT$. Endow $\bZ[w^\pm]$ with the natural signed graded ring structure. Thus we can talk about the top-degree specialization $(\mathrm{VP}(\Zmot_f(T)))^\td$. Note that there is an injective map of graded integral domains $\bZ[\cL^\pm]\hookrightarrow \bZ[w^\pm]$ sending $\cL$ to $w^2$. We extend this linearly to an injective ring map $\Phi:\bZ[\cL^\pm]\TT\to\bZ[w^\pm]\TT$. 
 We show that  
 \be\label{eqVP}
 (\mathrm{VP}(\Zmot_f(T)))^\td = \Phi(\rho(Z^\td_f(T))).
 \ee  
   In terms of the log resolution $\mu$, 
$$
(\mathrm{VP}(\Zmot_f(T)))^\td = \bigg(\sum_{I\subset S} (c_Iw^{2n}+\text{lower order terms in } w)\prod_{i\in I}\frac{(w^2)^{-\nu_i}T^{N_i}}{1-(w^2)^{-\nu_i}T^{N_i}}\bigg )^\td
$$
where $c_I$ is the number of irreducible components of $E_I^\circ$. Applying Lemma \ref{sum_of_td_part}, this further equals
\be\label{eqVPtd}
\bigg(\sum_{I\subset S} c_Iw^{2n}\prod_{i\in I}\frac{(w^2)^{-\nu_i}T^{N_i}}{1-(w^2)^{-\nu_i}T^{N_i}}\bigg)^\td.
\ee
In particular,   the orders of the poles of $\Zmot_f(T)$ can only decrease under the specialization $(\mathrm{VP}(\_))^\td$. Hence the  theorem  follows from Corollary \ref{poles_of_td_zeta} if we show (\ref{eqVP}).

On the other hand,
$$
\rho(Z^\td_f(T)) = \rho (\Zbir_{f,\mu}(T)^\td) = (\rho(\Zbir_{f,\mu}(T))^\td
$$
where the last equality holds 
since $\Zbir_{f,\mu}(T)$ has all coefficients subpositive, that is it lies in $ \cB_{\cL,\#}\TT$, and
 $\rho$ respects the signed graded ring structure. Now the claim follows since 
$$
\rho(\Zbir_{f,\mu}(T)) = \sum_{I\subset S} c_I\cL^{n}\prod_{i\in I}\frac{\cL^{-\nu_i}T^{N_i}}{1-\cL^{-\nu_i}T^{N_i}}
$$
which, after taking the top-degree subseries and applying $\Psi$, agrees with (\ref{eqVPtd}). We conclude the proof by applying Lemma \ref{rationality_non-positive} to $\mathrm{VP}(\Zmot_f(T))$. \hfill$\Box$

	\begin{corollary}
		If $-\alpha$ is a pole of $Z_{f}^{\mathrm{td}}(T)$ of order $n$, then $\alpha$ is the log canonical threshold of $f$ at some point $x\in f^{-1}(0)$.
	\end{corollary}
	\begin{proof}
		It follows directly from combining Theorem \ref{ThmA} (2) and \cite[Theorem 3.5]{NX16}.
	\end{proof}	
	
	\begin{subs}{\bf Proof of Theorem \ref{ThmAA} (2).} This follows from Theorem \ref{ThmA}. $\hfill\Box$ 
	\end{subs}

	\subs{\bf Proof of Lemma \ref{lemdlct}.} (1)  If one blows up a smooth irreducible subvariety $Z\subset Y$
that has only simple normal crossing with $E=\cup_{i\in S}E_i$ and $Z\subset E$, let $E_I$ with $I\subset S$ be the smallest stratum that contains $Z$. Then the corresponding exceptional divisor $E'$ has order of vanishing along $f$ equal to $N':=\sum_{i\in I}N_i>0$, and order of vanishing in the relative canonical divisor over $X$ equal to $\nu'=\sum_{i\in I}\nu_i+\codim(Z)-|I|$. Thus $\nu'/N'\ge \al_J$ and $N_J\mid N'$ for all $J\subset I$. Hence $\nu'/N'$ cannot dethrone any $\dlct(f)$ computed with $\mu$. Then the claim follows from the weak factorization theorem.

(2) To simplify the notation, in what follows we will only consider  non-empty $I,J\subset S$ with non-empty $E_I$, $E_J$. By definition, the set of remarkable numbers is the set $$\{\al_I\mid \al_J<\al_I\Rightarrow N_J\text{ does not divide } N_I\}.$$ If $\al_I$ is a remarkable number for an $I$ such that $\al_J<\al_I$ implies $N_J\nmid N_I$, then take $d=N_I$. Then $\al_I = \dlct(f)$. Indeed, if $N_J\mid d= N_I$, by the remarkable number condition we must have $\al_J\ge \al_I$.

Conversely, fix $d>0$ and write $\dlct(f)=\al_I$ for an $I$ with $N_I\mid d$. Suppose there exists $J$ with $\al_J<\al_I$ and $N_J\mid N_I$. Then $N_J\mid d$ and by the minimality condition in the definition of $\dlct(f)$ we must have $\al_J=\al_I$, a contradiction. Hence if $\al_J<\al_I$ then $N_J\nmid N_I$. Thus $\al_I$ is a remarkable number.
$\hfill\Box$

	\subs{\bf Proof of Theorem \ref{intrinsic_of_td_poles}.} 
    By definition, ${}^{(d)}\mathrm{lct}(f)$ exists if and only if $d$ is divisible by some $N_I$ with $E_I\neq \varnothing$. By \cite[Theorem A]{ELM04}, we have that the codimension of $\sX_m(f)$ is
	$$C_{lN+d}=\min_{I \in G_{lN+d}} \big\{\sum_{i\in I}a_i\nu_i\mid a\in\bN^I, \sum_{i\in I} a_iN_i=lN+d\big\}$$ where $G_k := \{I \subset S \mid  N_I \text{ divides }k\text{ and }E_I \neq \varnothing \}$. Since $N$ is divisible enough, we have $G_{lN+d} = G_d$ for all $l \in \mathbb N$. If the range of $a$ is non-empty, equivalently ${\sX}_m(f) \neq \varnothing$, then $N_I$ divides $d$. Conversely, if $N_I$ divides $d$, then for large $l$, there exists $I\in G_d$ and $a \in \mathbb N^I$ such that $\sum_{i\in I} a_iN_i = lN+d$. Hence, ${\sX}_m(f)$ is non-empty. These yield (1).

     For (2), we first prove that the sequence $\{C_{lN+d}(f)/(lN+d)\}_{l\gg 1}$ decreases. Let $i_1 \in I$ be such that $\nu_{i_1}/{N_{i_1}} = \min_{i\in I}\{\nu_i/N_i\}$. Define a new tuple $a'$ such that $a_i'=a_i$ if $i\neq i_1$ and $a_{i_1}'=a_{i_1}+N/N_{i_1}$. 	Then $\sum_{i\in I}a_i'N_i=(l+1)N+d$ and
		\begin{displaymath}			\frac{\sum_{i\in I} a_i'\nu_i}{(l+1)N+d} = \frac{\sum_{i\in I} a_i \nu _i + \frac{N\nu_{i_1}}{N_{i_1}}}{(lN+d) + N} \leq \frac{\sum_{i\in I}a_i\nu_i}{lN+d}.		\end{displaymath}		
		since this is equivalent to 
		\begin{equation}\label{eqonedotseven}
		    \frac{\nu_{i_1}}{N_{i_1}}\le \frac{\sum_{i\in I}a_i\nu_i}{\sum_{i\in I}a_iN_i},
		\end{equation}
		which follows from our assumption on $i_1$. Thus $C_{lN+d}/(lN+d)\ge C_{(l+1)N+d}/((l+1)N+d)$.

        Now we show $\lim_{l\to \infty} C_{lN+d}/(lN+d) = {}^{(d)}\mathrm{lct}(f)$. On the one hand, since $I \in G_d$, \eqref{eqonedotseven} implies that $C_{lN+d}/(lN+d) \geq {}^{(d)}\mathrm{lct}(f)$, and hence $\lim_{l\to \infty} C_{lN+d}/(lN+d) \geq  {}^{(d)}\mathrm{lct}(f)$. On the other hand, suppose $J \in G_d$ and $j_1\in J$ be such that ${}^{(d)}\mathrm{lct}(f) = \alpha_J = \nu_{j_1}/N_{j_1}$. Then for large $l_0$, there exists $a \in \mathbb N^J$ such that $\sum_{i\in J} N_j a_j = l_0N+d$. Now for $l > l_0$, take $b\in \mathbb N^J$ such that $b^l_i = a_{j_1}+(l-l_0)N/N_{j_1}$ if $ j = j_1$ and $b^l_i = a_{j}$ if $ j \neq j_1$. Then $C_{lN+d} \leq \sum_{j\in J} b_j^l\nu_j$ since $C_{lN+d}$ is the minimum. Consequently, we have $\lim_{l\to \infty} C_{lN+d}/(lN+d) \leq \lim_{l \to \infty} (\sum_{j\in J} b_j^l\nu_j)/(lN+d) = \nu_{j_1}/N_{j_1} = {}^{(d)}\mathrm{lct}(f)$.  
	$\hfill\Box$

	\subs{\bf Proof of Theorem \ref{intrinsic_of_td_poles_2}.}
		Fix a positive integer $m$. By shrinking $X$, we can further assume that $x$ is the unique singularity of $f$. We take $\mu$ to be an $m$-separated log resolution that is an isomorphism over $X\setminus\{x\}$. Let $S_{x} \subset S$ be the index set of exceptional divisors, $S_{m} := \{i \in S \mid N_i \text{ divides } m\}$, and $S_{m,x} := S_x\cap S_m$. Let $C_m$ be the codimension of $\sX_m(f)$ in the $m$-jet space of $X$.
		By Theorem \ref{structure_thm_of_contact_loci}, 
		$
			C_m= \min_{ i\in S_m}\{m\nu_i/N_i\} = \min\{\min_{i\in S_{m,x} }\{m\nu_i/N_i\}, m\}.
$		By \cite[Corollary 1.3]{McL19}, we have $v_m = \big(-2m-n+1+2\min_{i\in S_{m,x}}\{m\nu_i/N_i\}\big)/2m$, and $v_m$ is an embedded contact invariant of the link.
The summand $-2m$ does not appear in \cite[Corollary 1.3]{McL19}. This correction was pointed out in \cite[Remark 7.6]{BP}.
Therefore, 
$			C_m/m = \min\{v_m+\frac{n-1}{m}+1,1\}.
$		We  conclude the proof by applying Theorem \ref{intrinsic_of_td_poles}.
$\hfill\Box$




\subs{\bf Real remarkable numbers and the proof of Theorem \ref{ThmAAReal}.} We proceed with the notation and assumptions as in \ref{subsReal}.	We will use the terminology {\it real algebraic variety} in the sense of \cite{Man}. Let $K_0(\bR\Var)$ be the Grothendieck ring of real algebraic varieties. There are  morphisms of rings 
$$K_0(\Var_\bR)\xa{real} K_0(\bR\Var)\xa{\VP_\bR} \bZ[w].$$
The first morphism sends the class of  a reduced separated quasi-projective scheme $Z_\bR$ of finite type over $\bR$ to $Z_\bR(\bR)$ endowed with the natural real algebraic variety structure. The second morphism is the virtual Poincar\'e specialization of \cite{McP}, see also \cite[2.2]{Fic}.  If $Z$ is real algebraic variety, then $\VP_\bR(Z)$ is a polynomial in $w$ of degree equal to the dimension $d$ of $Z$, and the coefficient of $w^d$ is strictly positive by \cite{McP}. 

So we have the specializations
$\VP_\bR(real(\Zmot_{f_\bR}(T)))$ and $(\VP_\bR(real(\Zmot_{f_\bR}(T))))^\td$  in $\bZ[w^\pm]\TT$. Using the version of (\ref{Zmot}) for $\Zmot_{f_\bR}(T)$ and the log resolution $\mu_\bR$, we see that in the expression as a rational function of $real(\Zmot_{f_\bR}(T))$ only the classes of the real algebraic manifolds $E_I^\circ(\bR)\neq\emptyset$ with $I\subset S$ appear. Let $E_i'=E_i(\bR)$ as real algebraic manifold for $i\in S'$. Because $E_i$ are locally defined by algebraic local coordinates defined over $\bR$, we have $E_I^\circ(\bR)\neq\emptyset$ if and only if $I\subset S'$ and $(E'_I)^\circ:=\cap_{i\in I}E_i'\setminus\cup_{j\in S'\setminus I}E_j'\neq\emptyset$, in which case $E_I^\circ(\bR)=(E'_I)^\circ$ as real algebraic manifolds. Hence
$$
\VP_\bR(real(\Zmot_{f_\bR}(T))) = \sum_{\substack{I\subset S'\\E_I(\bR)\neq\emptyset}}(b_Iw^n+(\text{lower order terms in }w))\prod_{i\in I}\frac{w^{-\nu_i}T^{N_i}}{1-w^{-\nu_i}T^{N_i}}
$$
where $b_I>0$ if $E_I(\bR)\neq\emptyset$. Now we proceed with as in \ref{prfThmA1}, using the results from Section \ref{secTD} for rational quasi-subpositive series in $\bZ[w^\pm]\TT$, to conclude that the real remarkable numbers are exactly the poles of $(\VP_\bR(real(\Zmot_{f_\bR}(T))))^\td$, and thus they are intrinsic to $f$. Moreover, they are poles of $\VP_\bR(real(\Zmot_{f_\bR}(T)))$ by Lemma \ref{rationality_non-positive}. Since by specialization we can only loose poles, this implies that the real remarkable numbers are among the poles of $\Zmot_{f_\bR}(T)$. $\hfill\Box$


\section{Plane curves}\label{secCurves}

\subs{\bf Unibranch plane curve singularities.} 
Recall that a plane curve singularity $f: (\mathbb{C}^2,0) \to (\mathbb{C},0)$ is said to be \emph{unibranch} if its defining power series is irreducible. 
Any plane curve has a minimal resolution $\mu: Y \to X$ obtained by repeatedly blowing up the singular point until the total transform has simple normal crossings.
The datum of the minimal resolution of a unibranch plane curve is equivalent to several numerical invariants, such as the multiplicity sequence, the Puiseux pairs or the generators of the semigroup, see \cite[III.8.4]{BK}. We will assume that $f$ is not smooth.

For our computations, we will use yet another equivalent piece of data: the \emph{Newton pairs} of the branch.
This is a list of pairs of coprime integers $(\hkappa_1, \hr_1), \ldots, (\hkappa_g, \hr_g)$, and the sequence of blow-ups in the minimal resolution is encoded by the quotients that appear in the Euclidean algorithm for each pair, see \cite[\S 2.1]{BL} for the details.
We define $r_j \coloneqq \hr_j \cdots \hr_g$ and $\kappa_j \coloneqq \hkappa_j \hr_{j+1} \cdots \hr_g$.

In fact, we may associate to every exceptional component $E$ in the minimal resolution $\mu: Y \to X$ its own list of Newton pairs.
We define the Newton pairs of $E$ to be the Newton pairs of the blow-down $\mu_* \widetilde{D}$ of a curvette $\widetilde{D} \subset Y$ that intersects $E$ transversely and meets no other exceptional component, see \cite[Definition 2.6]{BL}.
We denote the Newton pairs of $E$ by $(\hkappa_1^E, \hr_1^E), \ldots, (\hkappa_{g_E}^E, \hr_{g_E}^E)$, and define $r_j^E \coloneqq \hr_j^E \cdots \hr_{g_E}^E$ and $\kappa_j^E \coloneqq \hkappa_j^E \hr_{j+1}^E \cdots \hr_{g_E}^E$. 
To make the notation more readable, we will write the last Newton pair as $(\hkappa^E, \hr^E) \coloneqq (\hkappa_{g_E}^E, \hr_{g_E}^E) = (\kappa_{g_E}^E, r_{g_E}^E)$.

 \begin{figure}[ht] 
 	\centering
 	\makebox[\textwidth][c]{\resizebox{0.7\linewidth}{!}{%
     \fontsize{20pt}{20pt}\selectfont
     \def\svgheight{0.5cm}
     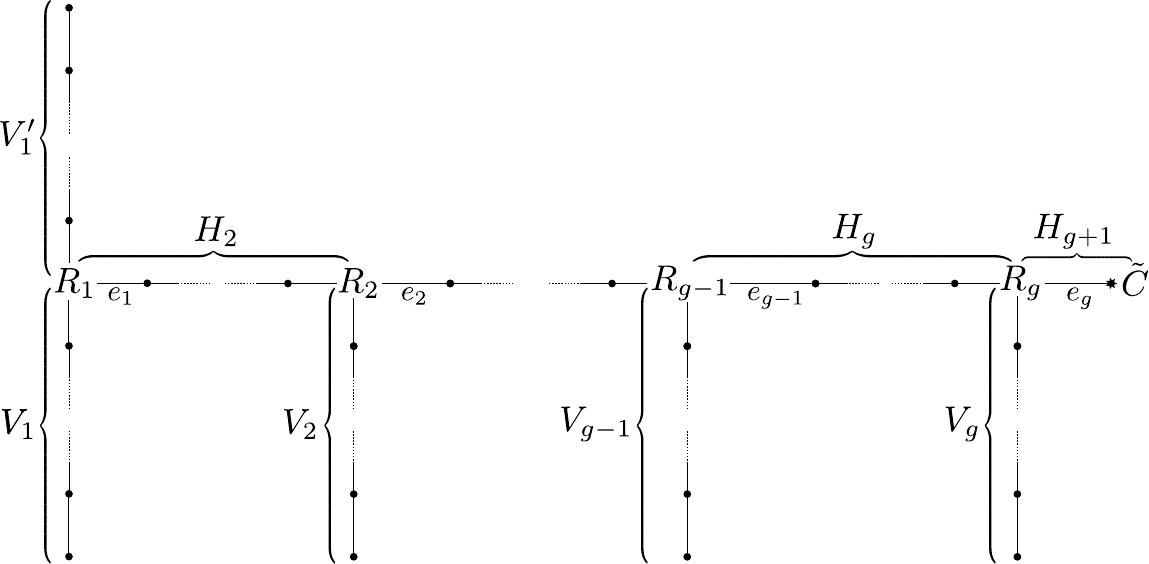
 }}
 	\caption{Partition of the resolution graph of a unibranch plane curve into groups.}
 	\label{fig:resolution-graph}
 \end{figure}

The dual graph of a unibranch plane curve always looks like Figure~\ref{fig:resolution-graph}.
Note that there are as many rupture components, that is, vertices of valency at least 3, as Newton pairs.
The rupture components divide the graph into horizontal groups $H_2, \ldots, H_{g+1}$ and vertical groups $V_1', V_1, V_2, \ldots V_g$.
Every non-rupture divisor belongs to exactly one of these groups, while each rupture divisor belongs to exactly three.

\begin{lemma} \label{lem:gcds}
\begin{enumerate}[wide]
	\item If $E, F$ are two adjacent divisors in $H_j$, then $\gcd(N_E, N_F) = r_j$.
	\item If $E, F$ are two adjacent divisors in $V_j$ for $j \geq 2$, then $\gcd(N_E, N_F) = N_{R_{j-1}} + \kappa_j$.
	\item If $E, F$ are two adjacent divisors in $V_1$, then $\gcd(N_E, N_F) = \kappa_1$.
	\item If $E, F$ are two adjacent divisors in $V_1'$, then $\gcd(N_E, N_F) = r_1$.
\end{enumerate}
\end{lemma}

\begin{proof}
	The fact that the gcds are constant in each group is an easy computation with intersection multiplicities, see \cite[Proposition 2.11]{BL}.
	To compute the actual value of the gcd in each case, we will use the following formula for the multiplicities, obtained from \cite[Proposition 2.20]{BL}:
	\begin{equation} \label{eq:multiplicities}
		N_E = \kappa_1^E r_1 + \cdots + \kappa_{g_E-1}^E r_{g_E-1} + \min\{ \kappa_{g_E}^E r_{g_E}, \kappa_{g_E} r_{g_E}^E \}.
	\end{equation}
	Observe that the minimum is attained by the first term if $E$ belongs to a horizontal group or to $V_1'$, and by the second term if $E$ belongs to a vertical group (different from $V_1'$).
	This is because the quotient $\hkappa^E / \hr^E$
	is strictly increasing as we move right in a horizontal group, and also as we move down in a vertical group, which follows from \cite[Proposition 2.5]{BL}.
	Furthermore, if $E = R_j$ is a rupture component, then $\hkappa^E / \hr^E = \kappa_{j} / r_{j}$ and thus both terms in the minimum are equal.
	Therefore,
	\[
		N_E = \begin{cases}
			\hr^E N_{R_{j-1}} + \hkappa^E r_j & \text{if } E \in H_j, \\
			\hr^E (N_{R_{j-1}} + \kappa_j) & \text{if } E \in V_j, j \geq 2, \\
			\hr^E \kappa_1 & \text{if } E \in V_1, \\
			\hkappa^E r_1 & \text{if } E \in V_1'.
		\end{cases}
	\]
	From here (2), (3) and (4) are obvious.
	To prove (1), an easy inductive argument shows that if $E, F$ are adjacent, then the matrix with columns $(\hkappa^E, \hr^E)$ and $(\hkappa^F, \hr^F)$ is invertible over $\mathbb Z$.
	Hence,
	\[
		\gcd(N_E, N_F) = \gcd(\hr^E N_{R_{j-1}} + \hkappa^E r_j, \hr^F N_{R_{j-1}} + \hkappa^F r_j) = \gcd(N_{R_{j-1}}, r_j) = r_j,
	\]
	because $r_j$ divides $N_{R_{j-1}}$ by formula \eqref{eq:multiplicities}.
\end{proof}

The next lemma is already known. It follows for example from \cite[Theorem 4.10]{NX16}. We give an elementary self-contained proof.

\begin{lemma} \label{lem:alpha-increase}
\begin{enumerate}[wide]
	\item On each horizontal group, $\alpha_E$ strictly increases as $E$ moves to the right.
	\item On each vertical group, $\alpha_E$ strictly increases as $E$ moves away from the rupture component.
\end{enumerate}
\end{lemma}

\begin{proof}
	An easy induction gives the following formula for the log-discrepancies, see the proof of \cite[Proposition 2.17]{BL}:
	\[
		\nu_E = \begin{cases}
			\hkappa^E + \hr^E \nu_{R_{j-1}} &\text{if } E \in H_j \text{ or } E \in V_j, j \geq 2;\\
			\hkappa^E + \hr^E &\text{if } E \in V_1 \text{ or } E \in V_1'.
		\end{cases}
	\]
	Combining this with the formulas for $N_E$ in the proof of Lemma~\ref{lem:gcds}, we see that the quotient $\alpha_E = \nu_E / N_E$ is increasing in $\hkappa^E/\hr^E$ if $E \in H_j$ or $E \in V_j, j \geq 1$, and decreasing in $\hkappa^E/\hr^E$ if $E \in V_1'$.
	Indeed, this is trivial for the vertical groups.
	For the horizontal groups, we compute
	\[
		\alpha_E = \dfrac{\hkappa^E + \hr^E \nu_{R_{j-1}}}{\hr^E N_{R_{j-1}} + \hkappa^E r_j} = \dfrac{t + \nu_{R_{j-1}}}{t r_j + N_{R_{j-1}}}
	\]
	where $t \coloneqq \hkappa^E/\hr^E$. 
	To see that this is increasing in $t$, we compute its derivative and note that $\frac{d}{dt} \alpha_E > 0$ if and only if $N_{R_{j-1}} > r_j \nu_{R_{j-1}}$. But we have
	\begin{align*}
		N_{R_{j-1}} &= \kappa_1^{R_{j-1}} r_1 + \cdots + \kappa_{j-1}^{R_{j-1}} r_{j-1}, \text{ and } \\
		r_j \nu_{R_{j-1}} &= r_j (r_1^{R_{j-1}} + \kappa_1^{R_{j-1}} + \cdots + \kappa_{j-1}^{R_{j-1}}) = r_1 + \kappa_1^{R_{j-1}} + \kappa_2^{R_{j-1}} r_j + \cdots + \kappa_{j-1}^{R_{j-1}} r_j,
	\end{align*}
	so the inequality follows from $r_i > r_j$ for $i < j$ and $\kappa_1^{R_{j-1}} r_1 > \kappa_1^{R_{j-1}} + r_1$ (because $\kappa_1^{R_{j-1}}$ and $r_1$ are integers greater or equal than 2).
	Together with the observation that $\hkappa^E / \hr^E$
	is strictly increasing as we move right in a horizontal group, and also as we move down in a vertical group, this proves the lemma.
\end{proof}

The following is equivalent to Theorem \ref{thmCurves}:

\begin{theorem} \label{thm:rmkble-unibranch}
	The remarkable numbers of a non-smooth unibranch plane curve singularity are precisely the $\alpha_{R_j} = \nu_{R_j} / N_{R_j}$ for $j = 1,\ldots, g$.
\end{theorem}

\begin{proof}
	For every $j = 1,\ldots,g$, let $e_j$ be the edge in the dual graph of the minimal resolution connecting $R_j$ with the divisor immediately to its right, see Figure~\ref{fig:resolution-graph}. 
	By Lemma~\ref{lem:alpha-increase}, $\alpha_{e_j} = \alpha_{R_j}$.
	Let us then show that $\alpha_{e_j}$ is remarkable.
	By Lemma~\ref{lem:alpha-increase}, all strata to the right of the edge $e_j$ have strictly larger $\alpha$-value, and hence we do not need to consider them.
	On the other hand, Lemma~\ref{lem:gcds} shows that $N_{e_j} = r_{j+1}$, and it also shows that $r_{j+1}$ divides $N_I$ for any stratum $I$ to the left of $e_j$.
	Hence, $\alpha_{e_j} = \alpha_{R_j}$ is remarkable.
	
	Finally, we show that there are no other remarkable numbers.
	To do this, let $I$ be any stratum different from the edges $e_j$.
	We need to show that either $I$ does not satisfy the condition in Definition~\ref{defRmkble}, or $\alpha_I = \alpha_{R_j}$ for some $j$:
	\begin{enumerate}
		\item If $I$ is contained in the horizontal group $H_j$ (and $I \neq e_{j-1}$), then $\alpha_I > \alpha_{e_{j-1}}$ by Lemma~\ref{lem:alpha-increase}, and $N_{e_{j-1}}$ divides $N_I$ by Lemma~\ref{lem:gcds}, so $I$ does not satisfy the condition in Definition~\ref{defRmkble}.
		\item If $I$ is contained in the vertical group $V_j$ (or $V_1'$ with $j=1$), then:
		\begin{enumerate}[label=(\alph*)]
			\item If $I$ is either $R_j$ or the edge connecting $R_j$ with the divisor immediately below it (or above it in the case of $V_1'$), then $\alpha_I = \alpha_{R_j}$.
			\item Otherwise, $\alpha_I > \alpha_{R_j} = \alpha_{e_j}$ by Lemma~\ref{lem:alpha-increase}, and $N_{e_j} = r_{j+1}$ divides $N_I$ by Lemma~\ref{lem:gcds}, so $I$ does not satisfy the condition in Definition~\ref{defRmkble}.
		\end{enumerate}
	\end{enumerate}
	This finishes the proof.
\end{proof}

	\begin{example}\label{exmany} Here is an explicit example.
	Let $k\ge 1$ and let $f_k \in \mathbb C[x,y]$ be a polynomial defining a unibranch plane curve singularity parameterized at the origin by $$x = t^{2^k}, \quad\quad y = t^{2^k+2^{k-1}} + t^{2^k+2^{k-1}+2^{k-2}} + \dots + t^{2^k+2^{k-1}+2^{k-2}+\ldots+1}.$$ 
		Let $X$ be a small neighborhood of the origin such that $f_k$ does not have other singularities in $X$. Let $f=f_k$ and let $\mu : Y \to X$ be the minimal log resolution of $f$.
The dual graph  can be computed using the algorithm of \cite[Theorem 15, p. 524]{BK} and is displayed as follows, where $\mu^*(\mathrm{div}(f)) = \sum_{i=1}^{2k+2} N_i E_i$ and  $K_{\mu} = \sum_{i=1}^{2k+2} (\nu_i-1) E_i$, and the number on an edge connecting $E_{i}$ and $E_j$ is $N_{\{i,j\}}=\gcd(N_i,N_j)$:
		\begin{displaymath}
			\xymatrix{
			{E_1} \ar@{-}^{2^k}[r] &	{E_3}	\ar@{-}^{2^{k-1}}[r] \ar@{-}_{3\cdot 2^{k-1}}[d] & {E_5}	\ar@{-}^{2^{k-2}}[r] \ar@{-}_{13\cdot 2^{k-2}}[d] & E_7 \ar@{-}_{53\cdot 2^{k-3}}[d] \ar@{-}[r] & \cdots \ar@{-}[r] & E_{2k-1} \ar@{-}^2[r] \ar@{-}_{\frac{5\cdot 2^{2k-2}-2}{3}}[d] & E_{2k+1} \ar@{-}^{1}[r] \ar@{-}_{\frac{5\cdot 2^{2k-1}-1}{3}}[d] & E_{2k+2}.\\
				& E_2 & E_4 & E_6 & & E_{2k-2} & E_{2k}
			}
		\end{displaymath}
The vanishing orders are:		
\renewcommand{\arraystretch}{1.2} 
		\begin{center}
			\begin{tabular}{|c|c|c|c|}
				\hline
				& $E_{2k+2}$  & $E_{2l}\,(1\leq l\leq k)$ & $E_{2l+1}\,(0\leq l\leq k)$\\
				\hline
				$N_i$ & $1$ & $2^{k-l}\cdot ({5\cdot 2^{2l-1}-1})/{3}$ & $2^{k-l+1}\cdot ({5\cdot 2^{2l-1}-1})/{3}$\\
				\hline
				$\nu_i$ & $1$ & $3\cdot 2^{l-1}$ & $3\cdot 2^l-1$\\
				\hline
			\end{tabular}
		\end{center}
There are then exactly $k$ remarkable numbers, given by $\nu_i/N_i$ with $i=3, 5, 7, \ldots, 2k+1$:
\renewcommand{\arraystretch}{1.3} 
\begin{center}
			\begin{tabular}{|c|c|}
				\hline
				$k$ & remarkable numbers of $f_k$  \\
				\hline
				$1$ & $\frac{5}{6}$ \\
				\hline
				$2$ & $\frac{5}{12}, \frac{11}{26}$ \\
				\hline
				$3$ & $\frac{5}{24},  \frac{11}{52},\frac{23}{106}$ \\
				\hline
				$4$ & $\frac{5}{48}, \frac{11}{104}, \frac{23}{212}, \frac{47}{426}$\\
				\hline
			\end{tabular}
		\end{center}
\end{example}

\begin{example}\label{exJN} Let $f=y^4+2x^3y^2+x^6+x^5y$. This defines the plane curve singularity $f_2$ with $k=2$ from Example \ref{exmany}. One can check using \cite{Sexp} that the remarkable number $11/26$ is does not appear in the Hodge spectrum  of $f$ at the origin. By \cite{Bu}, this also implies that $11/26$ is not a jumping number of $f$.
Using the minimal log resolution from above, we have that $\dlct(f)=5/12$ if $d$ is even, and $\dlct(f)=11/26$ if $d$ is odd, with $\dlct(f)$ as in the introduction.
\end{example}	

\begin{example} Consider the plane curve singularity $f_3$ from Example \ref{exmany}. The $d$-log canonical thresholds are  
$$\dlct(f)=\left\{
\begin{array}{cl}
\frac{23}{106} & \text{ if }d\equiv 1 \bmod 2,\\[2pt]
\frac{11}{52} & \text{ if }d\equiv 2 \bmod 4,\\[2pt]
\frac{5}{24} & \text{ if }d\equiv 0\bmod 4.
\end{array}
\right.$$ 
We observe a remarkable coincidence, pun intended. Let $\varphi$ be the action of the monodromy on $H^1$ of the Milnor fiber at the origin. Then the characteristic polynomial of $\varphi$ is $\Psi_{24}(t)\Psi_{52}(t)\Psi_{106}(t)$, a product of cyclotomic polynomials, which in general for an isolated hypersurface singularity can be computed by A'Campo's formula. Since the eigenvalues of $\varphi^d$ are $\lambda^d$ for $\lambda$ eigenvalue of $\varphi$, we can define the contributions $\tr(\varphi^d)=c_{24}(d)+c_{52}(d)+c_{106}(d)$ of these eigenvalues to the trace of  $\varphi^d$ by
$$
c_{k}(d):=\sum_{\substack{1\le a\le k\\ \gcd(a,k)=1}} e^{2\pi i da/k}.
$$
The remarkable coincidence is that the order $\frac{23}{106}>\frac{11}{52}>\frac{5}{24}$ between the remarkable numbers and their realizations as $\dlct(f)$ is reflected in the contributions to $\tr(\varphi^d)$:
$$\left\{
\begin{array}{cl}
c_{106}(d)\neq 0, c_{52}(d)=0, c_{24}(d)=0 & \text{ if }d\equiv 1 \bmod 2,\\
c_{106}(d)\neq 0, c_{52}(d)\neq 0, c_{24}(d)=0 & \text{ if }d\equiv 2 \bmod 4,\\
c_{106}(d)\neq 0, c_{52}(d)\neq 0, c_{24}(d)\neq 0 & \text{ if }d\equiv 0\bmod 4.
\end{array}\right.
$$
It would be interesting to understand this phenomenon in general. 
\end{example}

\begin{example}\label{exMultibr}
	The analog of Theorem~\ref{thm:rmkble-unibranch} does \emph{not} hold for multibranch plane curves.
	The reason for this is that, for each component $E$ in the minimal resolution graph, the multiplicity $N_E$ is computed as the sum of the multiplicities of each branch at $E$.
	Since the gcd does not behave well with respect to sums, the nice divisibility properties from Lemma~\ref{lem:gcds} no longer hold.
	Figure~\ref{fig:counterexamples} shows two counterexamples.
	Each node has been labeled with its last Newton pair, which is enough information to reconstruct the resolution process. 
	The rectangles indicate the rupture components that give remarkable numbers.
	The numbers and arrows in blue are extra data that has been added for the convenience of the reader: next to each node we write its $\alpha$-value as a nonreduced fraction $\nu_E / N_E$, and next to each edge we write the gcd of the multiplicities of its two adjacent components.
\end{example}

 \begin{figure}[ht] 
 	\centering
 	\makebox[\textwidth][c]{\resizebox{1.0\linewidth}{!}{%
     \fontsize{20pt}{20pt}\selectfont
     \def\svgheight{0.5cm}
     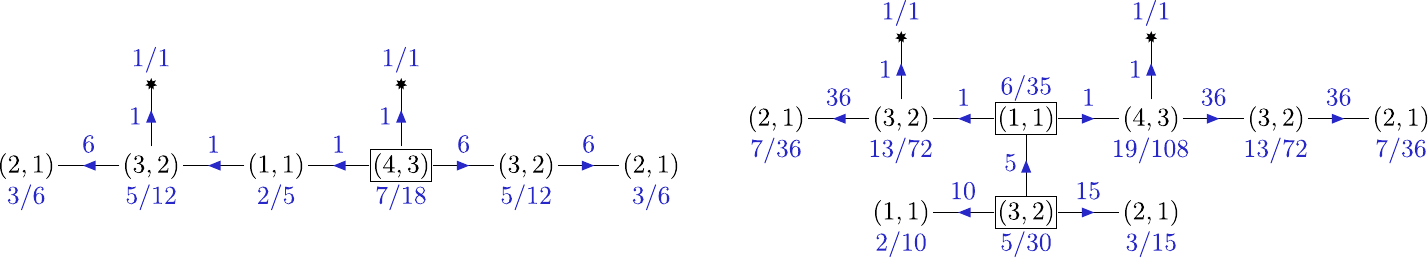
 }}
 	\caption{Two multibranch counterexamples to Theorem~\ref{thm:rmkble-unibranch}.}
 	\label{fig:counterexamples}
 \end{figure}

	\section{Examples}\label{Examples}

\subs{\bf Quasi-ordinary polynomials.} In many aspects quasi-ordinary hypersurface singularities  behave similarly to plane curve singularities. Starting from the canonical partial resolution from \cite[Theorem 3.7]{GV} in the unibranch case, one  derives the remarkable numbers in terms of the characteristic exponents using a computer:

\renewcommand{\arraystretch}{1.3} 
\begin{center}
			\begin{tabular}{|c|c|}
				\hline
				Characteristic exponents & Remarkable numbers  \\
				\hline
				$(\frac{1}{3},\frac{1}{3}), (\frac{2}{3},\frac{7}{6})$ from \cite[Example 9.1]{BGV} & $\frac{13}{22}$ \\
				\hline
				$(\frac{1}{2},\frac{1}{2},0), (\frac{2}{3},\frac{2}{3},\frac{11}{3})$ from \cite[Example 9.2]{BGV} & $\frac{14}{33}$ \\
				\hline
				$(\frac{1}{2},\frac{3}{2}), (\frac{1}{2},\frac{7}{4})$ from \cite[Example 8.14]{GP} & $\frac{5}{24}$, $\frac{11}{52}$ \\
				\hline
				$(\frac{1}{2},\frac{3}{2}), (\frac{1}{2},\frac{7}{4}), (\frac{1}{2},\frac{15}{8})$ & $\frac{5}{48}$, $\frac{11}{104}$, $\frac{23}{210}$ \\
				\hline
				$(\frac{1}{4},\frac{3}{2}), (\frac{1}{4},\frac{7}{4}), (\frac{1}{4},\frac{15}{8}), (\frac{1}{4},\frac{31}{16}), (\frac{1}{4},\frac{63}{32})$ & $\frac{5}{384}$, $\frac{11}{800}$, $\frac{23}{1608}$, $\frac{47}{3220}$, $\frac{95}{6442}$ \\
				\hline
			\end{tabular}
		\end{center}

\noindent The eigenvalue version of the monodromy conjecture for quasi-ordinary hypersurface singularities was proved by \cite{Artal}, \cite[Remark 3.18]{GV}. In particular, the eigenvalue part of Conjecture \ref{conjMCR} for the remarkable numbers holds in this case.



\subs{\bf Cones.}
One can also produce examples of hypersurfaces with non-isolated singularities in higher dimensions with other remarkable numbers than $\lct(f)$ 
by taking cones. For example, let $f = y^4z^2+2x^3y^2z+x^5y+x^6$ define the cone in $\bC^3$ over the projectivization of the plane curve from Example \ref{exJN}. Then $5/12$ and $11/26$ are the remarkable numbers of $f$ too.

\medskip

The next examples show the limitations of the concept of remarkable numbers.
	
	\subs{\bf Log canonical pairs.}\label{lcpair} 
	Let  $(X,D)$ be a log canonical pair, where $X$ is nonsingular, $D$ is a non-empty reduced divisor, say defined by a regular function $f$ on $X$. Then for any log resolution we have $N_i = 1$ for all the irreducible components $E_i$ of the strict transform of $D$. Hence $-1$ is  the only pole of $Z_f^{{td}}(T)$, cf. also \cite[4.2]{Birational zeta function}, and 1 is the only remarkable number.  	
	
	\subs{\bf Failure of Thom-Sebastiani.}\label{Thom-Sebastiani} The remarkable numbers do not always satisfy the Thom-Sebastiani rule. Let $h=f+g$, where $f=x^3+y^3$ and $g=t(z^2+t^3)$. Then $h$ has a unique singularity at the origin and is quasihomogeneous. One can compute that $\mathrm{lct}(f) = 2/3$, $\mathrm{lct}(g) = 5/8$, so they are remarkable numbers of $f$ and $g$, respectively. However, their sum $2/3+5/8=31/24$ is not a remarkable number of $h$. In fact, $\mathrm{lct}(h) = 1$ and it is the only remarkable number of $h$ by \ref{lcpair}.

	\subs{\bf Newton non-degenerate polynomials.}\label{exNondeg} Let $f = \sum_{\bm \alpha \in \mathbb N^n} c_{\bm \alpha}\bm x^{\bm \alpha} \in \mathbb C[\bm x] = \mathbb C[x_1,\dots,x_n]$ with $n \geq 2$, be a non-zero irreducible polynomial without  constant term. Here for $\bm \alpha\in \mathbb N^n$, we write $\bm x^{\bm \alpha}$ for $\prod_{i=1}^n x_i^{\alpha_i}$. Denote by $\mathrm{supp}(f) := \{\bm \alpha \in \mathbb N^n \mid c_{\bm \alpha} \neq 0\}$ the support of $f$. The \textit{Newton polyhedron} of $f$, denoted by $\Gamma_+(f)$, is the convex hull of $\cup_{\bm \alpha \in \mathrm{supp}(f)} (\bm \alpha + \mathbb N^n)$. The polyhedron $\Gamma(f)$ that is the union of all compact facets of $\Gamma_+(f)$, is called the \textit{Newton diagram} of $f$. For each facet $\gamma$ in $\Gamma(f)$, let $f_{\gamma} := \sum_{\bm \alpha \in \gamma} c_{\bm \alpha} \bm x^{\bm \alpha}$. Then $f$ is called \textit{Newton non-degenerate} if for all facets $\gamma$ in $\Gamma(f)$, the polynomials $f_{\gamma}$, $x_1\frac{\partial f_\gamma}{\partial x_1},\dots, x_n\frac{\partial f_{\gamma}}{\partial x_n}$ have no common zeros in $(\mathbb C^*)^n$.
		
	Recall that a cone $\tau \subset (\mathbb R_+)^n$ is \textit{unimodular} if $\tau = \sum_{i=1}^n \mathbb R_+ \cdot \bm t_i$ for some $\bm t_1,\dots,\bm t_n \in \mathbb Z^n \subset \mathbb R^n$ forming an integral basis of $\mathbb Z^n$. Then $\bm t_1,\dots,\bm t_n$ form the \textit{skeleton} of $\sigma$. Let $\gamma_1,\dots,\gamma_r$ be the facets in $\Gamma(f)$ and $R_1,\dots,R_r$ be the rays in $(\mathbb R_+)^n$ perpendicular to them. Then there exists a subdivision $\Sigma$ of $(\mathbb R_+)^n$ into unimodular cones such that each $R_i$ belongs to the skelton of a cone in $\Sigma$. 
	
	From now on we assume $f$ is Newton non-degenerate. Let $X(\Gamma(f))$ be the toric variety associated to $(\mathbb R^n,\Sigma)$ and $\pi : X(\Gamma(f)) \to \bA^n$ the toric morphism given by $(\mathbb R^n,\Sigma) \to (\mathbb R^n,\{(\mathbb R_+)^n\})$, see \cite{Ish14} for a definition. Then by \cite[10]{Var76}, $\pi$ is a log resolution of the germ of $f$ at the origin. For each cone $\sigma \in \Sigma$, one assigns an open subtorus $U_{\sigma} \simeq \mathbb A^n$ of $X(\Gamma(f))$. We describe $\pi$ on each $U_{\sigma}$.

	Let $\bm s_1,\dots,\bm s_n$ form the skeleton of $\sigma \in\Sigma$, and $\bm s_i = (s_{i1},\dots,s_{in})^t \in \mathbb N^n$. Then $U_\sigma$ is naturally identified with $\mathrm{Spec}\, \mathbb C[\bm x^{\bm \beta}]_{\bm \beta \in \sigma^\vee\cap \mathbb Z^n}$, where $\sigma^\vee = \{\bm v \in \mathbb R^n \mid \bm v \cdot \bm \alpha \geq 0 \text{ for all }\bm \alpha \in \sigma\}$ and $\bm v\cdot \bm \alpha$  means the standard inner product. The morphism $\pi_{\sigma} := \pi|_{\sigma} : U_{\sigma} \to \mathbb A^n_{\mathbb C}$ is given by the natural ring monomorphism $\mathbb C[\bm x] \to \mathbb C[\bm x^{\bm \beta}]_{\bm \beta \in \sigma^\vee\cap \mathbb Z^n}$. Let $(\bm g_1,\dots,\bm g_n)$ be the inverse matrix of $(\bm s_1,\dots,\bm s_n)$, then we have an isomoprhism $\mathbb C[\bm u] = \mathbb C[u_1,\dots,u_n] \overset{\sim}{\longrightarrow} \mathbb C[\bm x^{\bm \beta}]_{\bm \beta \in \sigma^\vee\cap \mathbb Z^n}$ given by $u_i \mapsto \bm x^{\bm g_i}$.
	
	The set of exceptional divisors of $\pi$ corresponds bijectively to $F_1$, the set of $1$-dimensional faces of cones in $\Sigma$ that are not standard coordinate rays. For example, for $\bm s_1$ above, the divisor correponding to $\mathbb R_+\cdot \bm s_1$ is expressed as $\{u_1 = \bm x^{\bm g_1} = 0\}$ in $U_{\sigma}$. Therefore, we have $\pi^*\mathrm{div}(f) = \widetilde D + \sum_{l\in F_1} N_lE_l$ and $K_\pi = \sum_{l\in F_1} (\nu_l-1) E_l$, where $\widetilde D$ is the strict transform of $\mathrm{div}(f)$ and $E_l$ is the exceptional divisor associated with the ray $l$. Here $N_l$ and $\nu_l$ are defined by $N_l = \mathrm{min}_{\bm w \in \Gamma(f)}\{\bm a\cdot \bm w\}$ and $\nu_l = a_1+\dots+a_n$, where $ l = \mathbb R_+\cdot \bm a$ with $a_1,\dots,a_n$ coprime.

	\begin{proposition} (Yifan Chen)
		Let $f \in  \mathbb C[x_1,...,x_n]\setminus\bC$ be a  Newton non-degenerate irreducible polynomial without constant term and $\Gamma(f) \neq \emptyset$. Let $E_{R_i}$ be an exceptional divisor corresonding to the perpendicular ray of a facet in $\Gamma(f)$, then $E_{R_i}\cap \widetilde D \neq \emptyset$. In particular, $-\mathrm{lct}(f)$ is the only pole of $Z_{f}^{{td}}(T)$ and $\lct(f)$ is the only remarkable number.
	\end{proposition}
	\begin{proof}
		Let $\sigma$ be a cone in $\Sigma$ having $R_i$ as a $1$-dimensional face. With the notation as above,  $f\circ \pi = \sum_{\bm \alpha} c_{\bm \alpha} \prod_{j=1}^n u_j^{\bm \alpha \cdot \bm s_j}$. Since $R_i$ is perpendicular to a facet in $\Gamma(f)$, there exists at least $n$ points $\bm p_1,\dots,\bm p_n \in \mathrm{supp}(f)$ such that $\bm p_i \cdot \bm s_1 = \min\{\bm \alpha \cdot \bm s_1 \mid c_{\bm \alpha} \neq 0\} = N_{R_1}$ for all $i=1,\dots,n$. Consequently  $f\circ \pi$ can be expressed as $u_1^{N_{R_1}}u_2^{q_2}...u_n^{q_n} \cdot \tilde f$ such that $u_1,...,u_n$ do not divide $\tilde f$ and $\{u_1 = 0\} \cap \{\tilde f= 0\} \neq \emptyset$. This proves the first assertion.

		Now we show the second assertion. If $\mathrm{lct}(f) = 1$ the claim follows from \ref{lcpair}. Otherwise, by the first assertion and Theorem \ref{poles_of_td_zeta}, it suffices to show that
		there exists $i\in\{1,\dots,n\}$ such that $\mathrm{lct}(f) = \nu_{R_i}/N_{R_i}$. Since  $-\mathrm{lct}(f)$ is the largest pole of  $\Zmot_f(T)$, this follows from the fact that the poles of the motivic zeta function of $\Zmot_f(T)$ are contained in $\{-1\} \cup \{-\nu_{R_i}/N_{R_i} \mid N_{R_i} > 0\}$, see \cite[Corollary 8.3.4]{BN20}.
	\end{proof}

\subs{\bf Hyperplane arrangements.} Let $f\in\bC[x_1,\ldots,x_n]$ be a reduced product of degree one polynomials. Then $\lct(f)$ is the only remarkable number. To see this, let $S$ be the set of all proper edges of $f$, that is, intersections of hyperplanes different than $\emptyset$ and $\bC^n$. Blowing up the (strict transforms of the) edges in increasing  dimension, one obtains a log resolution. In the notation of Definition \ref{defRmkble}, the subsets $I\subset S$ such that $E_I\neq\emptyset$ correspond to the flags $W_1\subsetneq\ldots\subsetneq W_{|I|}$ of edges in $S$, and for $W\in S$ we have $\nu_W=\codim W$ and $N_W$ is the number of hyperplanes in the arrangement containing $W$.  Take a flag containing an edge $W_0$ such that $\lct(f)=\nu_{W_0}/N_{W_0}$. Then that flag can be continued to end with a hyperplane of $f$, and we denote by $I\subset S$ the corresponding set of edges. Then $N_I=\gcd(N_{W'}\mid W'\in I)=1$ since $f$ is reduced. Hence there cannot be any other remarkable number   $>\lct(f)$.


\begin{thebibliography}{BBNV25 } 
	
	
	\bibitem[ACLM05]{Artal} E. Artal Bartolo, P. Cassou-Noguès, I. Luengo, A. Melle Hernández. Quasi-ordinary power series and their zeta functions. Mem. Amer. Math. Soc. 178, vi+85 pp. (2005). 
		
		\bibitem[BBNV25]{Birational zeta function}
		T. Biesbrouck, N. Budur, J. Nicaise,  W. Veys.
		\newblock Birational zeta functions.
		\newblock{arxiv:2509.03352}.
		
		
		
		\bibitem[BK86]{BK} E. Brieskorn, H. Kn\"orrer. 
{Plane algebraic curves.} Birkh\"auser/Springer,  x+721 pp. (1986).
		
		\bibitem[B03]{Bu} N. Budur. On Hodge spectrum and multiplier ideals.  Math. Ann.  327,  257-270 (2003). 


\bibitem[B12]{Bsurvey} N. Budur. Singularity invariants related to Milnor fibers: survey. Contemp. Math. 566, 161-187 (2012).

\bibitem[BGG12]{BGV}	N. Budur, P. González-Pérez, M. González Villa. 
Log canonical thresholds of quasi-ordinary hypersurface singularities.
Proc. Amer. Math. Soc. 140, 4075-4083 (2012).
				
		\bibitem[B+24]{BBPP24}
		N. Budur, J. de la Bodega, E. de Lorenzo Poza, J. Fern\'andez de Bobadilla, T. Pe\l ka.
		\newblock{On the embedded Nash problem}.
		\newblock{Forum Math. Pi} 12, Paper No. e15, 28 pp. (2024).
		
		
		\bibitem[BFLN22]{Cohomology_of_Contact_Loci}
		N. Budur, J.~Fern\'andez de~Bobadilla, Q.T. L\^e,  H.D.
		Nguyen.
		\newblock Cohomology of contact loci.
		\newblock {J. Differential Geom.} 120, 389-409 (2022).
		


		\bibitem[BSZ24]{BSZ} N. Budur, Q. Shi,  H. Zuo. \newblock{Motivic principal value integrals for hyperplane arrangements}. \newblock{ arXiv:2411.01305.}
	 	
		\bibitem[BN20]{BN20}
		E. Bultot, J. Nicaise.
		\newblock{ Computing motivic zeta functions on log smooth models.}
		\newblock{ Math. Zeit.} 295, 427-462 (2020).

		\bibitem[BL23]{BL}
		J. de la Bodega, E. de Lorenzo Poza.
		\newblock{The arc-Floer conjecture for plane curves}.
		\newblock{arXiv:2308.00051}. To appear in J. Differential Geom.
		
		\bibitem[DL98]{DL98}
		J. Denef, F. Loeser.
		\newblock Motivic {I}gusa zeta functions.
		\newblock { J. Algebraic Geom.} 7, 505-537 (1998).
		
		
		\bibitem[ELM04]{ELM04}
		L. Ein, R. Lazarsfeld, M. Musta\c t\u a.
		\newblock{ Contact loci in arc spaces.}
		\newblock{ Compos. Math.} 140, 1229-1244 (2004).
		
		
		
	\bibitem[FP24]{BP} J. Fern\'andez de Bobadilla, T. Pe\l ka. Symplectic monodromy at radius zero and equimultiplicity of $\mu$-constant families. { Ann. Math.} 200, 153-299 (2024).
	
	
	\bibitem[Fi12]{Fic} G. Fichou. The motivic real Milnor fibres. Manuscripta Math. 139, 167–178 (2012).
		
		
		
\bibitem[GG14]{GV}  P. González Pérez, M. González Villa.
Motivic Milnor fiber of a quasi-ordinary hypersurface. J. Reine Angew. Math. 687, 159-205 (2014).		
		
\bibitem[GR25]{GP} P. González Pérez, M. Robredo Buces. Quasi-ordinary hypersurfaces, multiplier ideals and local tropicalizations. arXiv:2510.21009.

		
		\bibitem[I18]{Ish14}
		S. Ishii.
		\newblock{Introduction to singularities}.
		\newblock Springer (2018).
		
		\bibitem[KT19]{KT}
		M. Kontsevich, Y. Tschinkel.
		\newblock{Specialization of birational types.}
		\newblock{ Invent. Math.} 217, 415-432 (2019).
		
\bibitem[L88]{L} F. Loeser. Fonctions d'Igusa $p$-adiques et polyn\^{o}mes de
Bernstein, Amer. J. Math. 110, 1-21 (1988).



		
			\bibitem[Ma20]{Man} F. Mangolte. Real algebraic varieties
Springer, 2020, xviii+444 pp.	

\bibitem[McCP03]{McP} C. McCrory, A. Parusi\'nski. Virtual Betti numbers of real algebraic varieties. C. R. Math. Acad. Sci. Paris 336, 763–768 (2003).
		
				\bibitem[McL19]{McL19}
		M. McLean.
		\newblock Floer cohomology, multiplicity and the log canonical threshold.
		\newblock { Geom. Topol.} 23, 957-1056 (2019).	
				
		
		
		
				
		\bibitem[Mu02]{Mus02}
		M. Musta\c t\u a.
		\newblock{ Singularities of pairs via jet schemes.}
		\newblock{ J. Amer. Math. Soc.} 15, 599-615 (2002).
		
		\bibitem[NX16]{NX16}
		J. Nicaise, C. Xu.
		\newblock Poles of maximal order of motivic zeta functions.
	Duke Math. J. 165, 217-243 (2016).
		
		\bibitem[RV03]{RV} B. Rodrigues, W. Veys.
Poles of zeta functions on normal surfaces.
{ Proc. London Math. Soc.} 87, 164-196 (2003).

\bibitem[S00]{Sexp} M. Saito. Exponents of an irreducible plane curve singularity. arXiv:math/0009133.

		
	\bibitem[S07]{S-real} M. Saito.  On real log canonical thresholds. arXiv:0707.2308.
	
		
		\bibitem[Va76]{Var76}
		A. Varchenko.
		\newblock Zeta-function of monodromy and Newton's diagram.
		\newblock{ Invent. Math.} 37, 253-262 (1976).
		
\bibitem[Va82]{Va} A. Varchenko. The complex singularity index does not change along the stratum $\mu=$const. Funct. Anal. Appl. 16, 1-9  (1982).		
		
\bibitem[Ve95]{Ve} W. Veys. Determination of the poles of the topological zeta function for curves. Manuscripta
Math 87, 435-448 (1995).		
				
	\end{thebibliography}
\end{document}